\documentclass[10pt]{amsart}
\usepackage{amsfonts}
\usepackage{amssymb}
\usepackage{amscd}
\usepackage{amsthm}
\usepackage[dvips]{graphicx}
\usepackage[all,cmtip]{xy}
\usepackage{hyperref}
\usepackage{yhmath}

\usepackage{ulem}
\usepackage[all]{xy}
\usepackage{mathdots}
\usepackage{leftidx}
\usepackage{mathrsfs}
\usepackage{amsmath,amssymb, amsthm}
\usepackage{color}
\usepackage{tikz}
\usepackage{picture}
\usepackage{enumerate}
\usepackage{makecell}
\usepackage{extarrows}
\usepackage{graphicx}
\usepackage{caption}
\usepackage{subcaption}
\usepackage{float}
\numberwithin{equation}{section}
\newtheorem{Thm}{Theorem}[section]
\newtheorem{Lem}[Thm]{Lemma}
\newtheorem{Def}[Thm]{Definition}
\newtheorem{Cor}[Thm]{Corollary}
\newtheorem{Prop}[Thm]{Proposition}

\newtheorem{Rem}[Thm]{Remark}

\title[ on simple-minded systems]{ on simple-minded systems over representation-finite self-injective algebras}
\author{Jing Guo, Yuming Liu, Yu Ye and Zhen Zhang}
\address{Jing Guo
\newline School of Mathematical Sciences
\newline University of Science and Technology of China
\newline Hefei, Anhui 230026
\newline P.R.China}
\email{gjws@mail.ustc.edu.cn}
\address{Yuming Liu
\newline School of Mathematical Sciences
\newline Laboratory of Mathematics and Complex Systems
\newline Beijing Normal University
\newline Beijing 100875
\newline P.R.China}
\email{ymliu@bnu.edu.cn}
\address{Yu Ye
\newline School of Mathematical Sciences
\newline Wu Wen-Tsun Key Laboratory of Mathematics
\newline University of Science and Technology of China
\newline Hefei, Anhui 230026
\newline P.R.China}
\email{yeyu@ustc.edu.cn}
\address{Zhen Zhang
\newline School of Mathematical Sciences
\newline Beijing Normal University
\newline Beijing 100875
\newline P.R.China}
\email{zhangzhen@mail.bnu.edu.cn}

\date{version of \today}

\newenvironment{Proof}[1][Proof]{\begin{trivlist}
\item[\hskip \labelsep {\bfseries #1}]}{\flushright
$\Box$\end{trivlist}}

\newcommand{\lra}{\longrightarrow}

\newcommand{\ra}{\rightarrow}
\newcommand{\sdp}{\times\kern-.2em\vrule height1.1ex depth-.05ex}
\newcommand{\epi}{\lra \kern-.8em\ra}

\newcommand{\Z}{{\mathbb Z}}

\newcommand{\stmod}{\underline{\mathrm{mod}}}
\newcommand{\Hom}{\mathrm{Hom}}
\newcommand{\ind}{\mathrm{ind}}
\newcommand{\stHom}{\underline{\mathrm{Hom}}}
\newcommand{\stind}{\underline{\mathrm{ind}}}
\newcommand{\Supp}{\mathrm{Supp}}

 \normalbaselines
\begin{document}
\renewcommand{\thefootnote}{\alph{footnote}}
\renewcommand{\thefootnote}{\alph{footnote}}
\setcounter{footnote}{-1} \footnote{\it{Mathematics Subject
Classification(2010)}$\colon$16G20, 18Gxx.}
\renewcommand{\thefootnote}{\alph{footnote}}
\setcounter{footnote}{-1}
\footnote{\it{Keywords}$\colon$Simple-minded system; RFS algebra; Orthogonal system; Nakayama-stable.}

\begin{abstract}
Let $A$ be a representation-finite self-injective algebra over an algebraically closed field $k$. We give a new characterization for an orthogonal system in the stable module category $A$-$\stmod$ to be a simple-minded system. As a by-product, we show that every Nakayama-stable orthogonal system in $A$-$\stmod$ extends to a simple-minded system.

\end{abstract}
\maketitle

\section{Introduction}

As an attempt towards a tilting theory for stable equivalences between finite dimensional algebras, Koenig and Liu \cite{KL} introduced simple-minded systems in the stable module category $A$-$\stmod$ of any finite dimensional algebra $A$. Roughly speaking, a simple-minded system over $A$ is a family of objects in $A$-$\stmod$ which satisfies orthogonality and a generating condition. Later, Dugas \cite{Dugas} defined simple-minded systems in any Hom-finite Krull-Schmidt triangulated category. The two definitions are equivalent in the stable module category of a self-injective algebra (cf. \cite[Section 2.1]{CLZ}).
In \cite{C}, Coelho Sim\~{o}es introduced simple-minded systems in $-d$-Calabi-Yau triangulated categories for any negative integer $d$. There is a recent rise of interests in studying $d$-simple-minded systems (see, for example, \cite{CP,CPP,IJ}).
 
On the other hand, Chan, Koenig and Liu \cite{CKL} noticed  that for a representation-finite self-injective algebra $A$, the simple-minded systems in $A$-$\stmod$ correspond exactly to the combinatorial configurations in the Auslander-Reiten quiver of $A$, a key notion introduced by Riedtmann (\cite{Riedtmann2,Riedtmann3,Riedtmann4}) in the 1980's in her famous work on classification of representation-finite self-injective
algebras. A similar notion in $-d$-Calabi-Yau triangulated categories is called $d$-Riedtmann configuration (see \cite{CP}) or $-d$-Calabi-Yau configuration (see \cite{J1}). The connection between simple-minded systems and combinatorial configurations is quite useful since the combinatorial configurations are often easier to handle.
	

In general, it is hard to check the two conditions in the definition of a simple-minded system.
So it is important to find easier characterizations of simple-minded systems. In this paper, we will give such a characterization  of simple-minded systems over representation-finite self-injective algebras.
Before stating our result, we recall some notations and results from \cite{CKL}. Let $A$ be an RFS algebra (that is, indecomposable, basic representation-finite self-injective algebra ($\ncong k$) over an algebraically closed field $k$) and $\mathcal{S}$ a simple-minded system in $A$-$\stmod$.
Then, according to \cite{CKL}, $\mathcal{S}$ is an orthogonal system (see Definition \ref{brick-orthogonal-system} ) in $A$-$\stmod$ (orthogonality condition), the cardinality of $\mathcal{S}$ is equal to the number of non-isomorphic simple $A$-modules (cardinality condition), and the Nakayama functor on $A$-$\stmod$ permutes the objects of $\mathcal{S}$ (Nakayama-stable condition).
The  main result in this paper says that the above three conditions are also sufficient for a family of objects in $A$-$\stmod$ to be a simple-minded system.

\begin{Thm} Let $A$ be an RFS algebra and $\mathcal{S}$ a family of objects in $A$-$\stmod$. Then $\mathcal{S}$ is a simple-minded system if and only if $\mathcal{S}$ satisfies the following three conditions.
\begin{enumerate}[$(1)$]
\item $\mathcal{S}$ is an orthogonal system in $A$-$\stmod$.
\item The cardinality of $\mathcal{S}$ is equal to the number of non-isomorphic  simple $A$-modules.
\item $\mathcal{S}$ is Nakayama-stable, that is, the Nakayama functor on $A$-$\stmod$ permutes the objects of $\mathcal{S}$.
\end{enumerate}
\end{Thm}

There are two main ingredients in the proof of the above theorem. One is the torsion pair theory studied by Iyama-Yoshino \cite{IY} and by Dugas \cite{Dugas}. The other one is the covering theory developed by Riedtmann \cite{Riedtmann2} and by Bongartz-Gabriel \cite{BG}. From the proof of Theorem 1.1, we also deduce some new properties of orthogonal systems in $A$-$\stmod$. In particular, we prove the following extendible property of Nakayama-stable orthogonal systems in $A$-$\stmod$.

\begin{Thm} Let $A$ be an RFS algebra. Then every Nakayama-stable orthogonal system $\mathcal{S}$ in $A$-$\stmod$ extends to a simple-minded system.
\end{Thm}

This paper is organized as follows. In Section 2,
we recall some notions and facts on torsion pair theory, covering theory and simple-minded systems.
In Subsection 3.1, we prove our main result Theorem 1.1 and give some applications. Our proof here is based on three technical lemmas$\colon$Lemma \ref{Nakayaka-stable-conclusion-2}, Lemma \ref{two-torsion-pairs}, and Lemma \ref{two-side-orthogonal}. The first two lemmas come from torsion pair theory and the last one relies on several interesting orthogonality properties in the stable module categories of several classes of RFS algebras (Lemma \ref{Dn-En-stable-brick} to Lemma \ref{orthogonal-nu-orbit-of-D_{3m}}). In Subsection 3.2, we prove Theorem 1.2 and its corollary.

\section*{ Acknowledgements} The authors are supported by NSFC (No.11331006, No.11431010, No.11571329 and  No.11971449). We would like to thank Steffen Koenig and Aaron Chan for comments and many suggestions on the presentation of this paper.  We are grateful to the referee for valuable comments and suggestions.

\section{Preliminaries}

Throughout this paper, $k$ denotes an algebraically closed field, all algebras are assumed to be finite dimensional
$k$-algebras with $1$. For an algebra $A$, we denote by $A$-mod the category of finite dimensional left $A$-modules, and by $A$-$\stmod$ the stable category of $A$-mod, that is, the category with the same class of objects but with morphism spaces $\stHom_A(X,Y)$ being the quotient of the ordinary one $\mathrm{Hom}_A(X,Y)$ by maps factoring through projective modules.

\subsection{Torsion pair theory }

We briefly recall the torsion pair theory on a Hom-finite Krull-Schmidt triangulated
$k$-category in the sense of Dugas \cite{Dugas}. Let $\mathcal{T}$ be a Hom-finite Krull-Schmidt triangulated
$k$-category with suspension functor $[1]$. For any families $\mathcal{S}_{1}, \mathcal{S}_{2}$ of objects in $\mathcal{T}$, we
define a family of objects
$$\mathcal{S}_{1}\ast\mathcal{S}_{2}:=\{ X\in \mathcal{T}\mid \mbox{ There is a distinguished triangle }S_{1} \longrightarrow  X \longrightarrow S_{2} \longrightarrow S_{1}[1],  S_{1}\in \mathcal{S}_{1}, S_{2}\in \mathcal{S}_{2}\}. $$
 Using the octahedral axiom, it is easy to show  that $(\mathcal{S}_{1}\ast\mathcal{S}_{2})\ast\mathcal{S}_{3}=\mathcal{S}_{1}\ast(\mathcal{S}_{2}\ast\mathcal{S}_{3})$
for $\mathcal{S}_{1}, \mathcal{S}_{2}, \mathcal{S}_{3}\subseteq\mathcal{T}$. For a family $\mathcal{S}$ of objects in $\mathcal{T}$,
we denote $(\mathcal{S})_{0}=\{0\}$, and for any positive integer $n$,
we inductively define $(\mathcal{S})_{n}=(\mathcal{S})_{n-1}\ast(\mathcal{S}\cup\{0\})$.
$(\mathcal{S})_{n}\ast(\mathcal{S})_{m}=(\mathcal{S})_{n+m}$ for any non-negative integers $m$ and $n$ (cf. \cite[Lemma 2.3]{Dugas}).
Similarly, one can define $ _{n}(\mathcal{S})$, and we have $(\mathcal{S})_{n}$=$_{n}(\mathcal{S})$. We say that $\mathcal{S}$
is {\it extension-closed}, if $\mathcal{S}\ast\mathcal{S}\subseteq \mathcal{S}$. One denotes the extension closure of a family
$\mathcal{S}$ of objects in $\mathcal{T}$ as
$$\mathcal{F}(\mathcal{S}):=\bigcup_{n\geq0}(\mathcal{S})_{n},$$
which is the smallest extension closed full subcategory of $\mathcal{T}$ containing $\mathcal{S}$. Notice that we identify $\mathcal{S}$
with the corresponding full (usually not triangulated) subcategory of $\mathcal{T}$.

\begin{Def}\label{brick-orthogonal-system}
  An object $M$ in $\mathcal{T}$ is a stable brick if $\mathcal{T}(M,M)\cong k$.   Moreover, a family $\mathcal{S}$ of stable bricks in $\mathcal{T}$ is an orthogonal system if $\mathcal{T}(M,N)=0$ for all distinct  $M, N$ in $\mathcal{S}$.
\end{Def}

\begin{Lem}\label{orthogonal-bricks-closed-summands}{\rm(\cite[Lemma 2.7]{Dugas})} If $\mathcal{S}\subseteq \mathcal{T}$ is an orthogonal system, then $(\mathcal{S})_{n}$ is closed under direct summands for each positive integer $n\geq1$. In particular, $\mathcal{F}(\mathcal{S})$ is closed under direct summands.
\end{Lem}

For any family $\mathcal{S}$ of objects in $\mathcal{T}$, we set
$$\mathcal{S^{\perp}}:=\{Y\in \mathcal{T}\mid\mathcal{T}(X,Y)=0, \forall X\in \mathcal{S}\},$$
$$\mathcal{^{\perp}S}:=\{Y\in \mathcal{T}\mid\mathcal{T}(Y,X)=0, \forall X\in \mathcal{S}\}.$$
We know that both $\mathcal{S^{\perp}}$ and $\mathcal{^{\perp}S}$ are extension closed subcategories of $\mathcal{T}$  as well as closed under direct summands. We shall denote $\mathcal{S^{\perp}}\cap\mathcal{^{\perp}S}$ by $\mathcal{^{\perp}S^{\perp}}$.

\begin{Def}\label{torsion-pair}{\rm(\cite[Definition 3.1]{Dugas})} A pair $(\mathcal{X},\mathcal{Y})$ of full, additive subcategories of $\mathcal{T}$, which are closed under direct summands, forms a torsion pair if the following conditions hold$\colon$
\begin{enumerate}[$(1)$]
\item $\mathcal{T}(\mathcal{X}, \mathcal{Y})=0$.
\item $\mathcal{T}=\mathcal{X}\ast\mathcal{Y}$, that is, for each $T\in\mathcal{T}$, there exists a distinguished triangle
   $$\xymatrix@C=0.5cm{
   X \ar[r]^{f} & T\ar[r]^{g} &Y \ar[r]^{} & X[1]},$$
    where $X\in\mathcal{X}, Y\in\mathcal{Y}.$
\end{enumerate}
\end{Def}

The above distinguished triangle in $(2)$ is called a {\it $(\mathcal{X},\mathcal{Y})$-triangle} of $\mathcal{T}$. It is easy to show that for any $(\mathcal{X},\mathcal{Y})$-triangle of $\mathcal{T}$, $f$ is a right $\mathcal{X}$-approximation and $g$ is a left $\mathcal{Y}$-approximation. It is true that for a $(\mathcal{X},\mathcal{Y})$-triangle, $f$ is a minimal right $\mathcal{X}$-approximation if and only if $g$ is a minimal left $\mathcal{Y}$-approximation (cf. \cite[Lemma 3.2]{Dugas}). Furthermore, we can choose a right minimal version of $f$ and this resulting triangle is unique up to isomorphism, we call it the  {\it minimal $(\mathcal{X}, \mathcal{Y})$-triangle}.

In the present paper, we will apply the above torsion pair theory in a special case where $\mathcal{T}$ is the stable category $A$-$\stmod$ of a self-injective algebra $A$. In this case, the suspension functor is the cosyzygy functor
$\Omega^{-1}$ (sometimes still denoted by $[1]$ if there is no confusion) and the distinguished triangles in $A$-$\stmod$ are
induced by short exact sequences in $A$-mod (see \cite{H}). Notice that $A$-$\stmod$ has Serre functor $\nu\Omega=\nu[-1]$, that is, for all $M,N\in A$-$\stmod$, we have the natural $k$-linear isomorphisms$\colon\stHom_A(M,N) \cong D\stHom_A(N,\nu\Omega M)$, where $D=\Hom_k(-,k)$ is the usual $k$-dual functor and $\nu=D\Hom_A(-,A)$ is the Nakayama functor (see \cite{RV}). We remind the reader that the Nakayama functor defines a self-equivalence on $A$-mod (hence on $A$-$\stmod$).

Now we take an orthogonal system $\mathcal{S}$  in $A$-$\stmod$ with $\nu(\mathcal{S})=\mathcal{S}$ and assume that both $(\mathcal{^\perp S}, \mathcal{F}(\mathcal{S}))$ and $(\mathcal{F}(\mathcal{S}),\mathcal{S^{\perp}})$ are torsion pairs in $A$-$\stmod$. Let $X$ be an object in $A$-$\stmod$.  We define operators $\mathbf{a}\colon\mathcal{T}\longrightarrow \mathcal{^{\perp}S}, \mathbf{b}, \mathbf{c}\colon\mathcal{T}\longrightarrow \mathcal{F}(\mathcal{S})$ and $\mathbf{d}\colon\mathcal{T}\longrightarrow \mathcal{S^{\perp}}$ via the minimal triangles  $$
   \mathbf{a}X\longrightarrow X\longrightarrow \mathbf{b}X\longrightarrow \mbox{ and } \mathbf{c}X\longrightarrow X\longrightarrow \mathbf{d}X\longrightarrow $$
corresponding to these two torsion pairs respectively. Notice that in general these operators are not functors, see \cite[Section 3]{Dugas} for more information.

\begin{Lem}\label{Nakayaka-stable-conclusion-1}{\rm(\cite[Lemma 4.3]{Dugas})} Assume that $\mathcal{S}$ is an orthogonal system in $A$-$\stmod$ with $\nu(\mathcal{S})=\mathcal{S}$. Then $\nu(\mathcal{F}(\mathcal{S}))=\mathcal{F}(\mathcal{S})$. Furthermore,  $\nu(\mathbf{a}X)\cong\mathbf{a}(\nu X)$ and $\nu(\mathbf{b}X)\cong\mathbf{b}(\nu X)$ for all $X\in$ $A$-$\stmod$, and similarly for $\mathbf{c}$ and $\mathbf{d}$.
\end{Lem}

\begin{Lem}\label{Nakayaka-stable-conclusion-2}{\rm(\cite[Lemma 4.6]{Dugas})}  Let $\mathcal{S}$ be as in Lemma {\rm\ref{Nakayaka-stable-conclusion-1}}. For any minimal $(\mathcal{^{\perp}S},\mathcal{F}(\mathcal{S}))$-triangle $\mathbf{a}Y\stackrel{f}{\longrightarrow} Y\stackrel{g}{\longrightarrow} \mathbf{b}Y\longrightarrow $ and any $X\in \mathcal{S}$, we have the following.
\begin{enumerate}[$(1)$]
\item The map $\stHom_{A}(g,X)\colon\stHom_{A}(\mathbf{b}Y,X)\longrightarrow \stHom_{A}(Y,X)$ is an isomorphism.
\item  The map $\stHom_{A}(X,f)\colon \stHom_{A}(X,\mathbf{a}Y)\longrightarrow\stHom_{A}(X,Y)$ is a monomorphism.
\item If $Y\in \mathcal{S^{\perp}}$, then $\mathbf{a}Y\in \mathcal{^{\perp}S^{\perp}}$.
\end{enumerate}
\end{Lem}

\subsection{Covering theory}

 The covering of translation quivers was introduced by Riedtmann (\cite{Riedtmann1}), and it was extended to covering functors between $k$-categories by Bongartz and Gabriel (\cite{BG}). We refer to a brief introduction on some covering theory from \cite{CKL}.

Following Asashiba \cite{A1}, we abbreviate (indecomposable, basic) representation-finite self-injective algebra ($\ncong k$) over an algebraically closed field $k$ by RFS algebra. Let $A$ be an RFS algebra, and let $_{s}\Gamma_{A}$ be the stable Auslander-Reiten quiver of $A$. It is known that $_s\Gamma_A$ has the form
$\mathbb{Z}\Delta/\langle\sigma\tau^{-r}\rangle$,
where $\Delta$ is a Dynkin quiver, $\mathbb{Z}\Delta$ is the stable translation quiver associated to $\Delta$, $\tau$ is the translation of $\mathbb{Z}\Delta$ and $\sigma$ is some automorphism of the quiver $\mathbb{Z}\Delta$ with a fixed vertex. Notice that $\tau$ coincides with the AR-translate $DTr$.
According to \cite{A2}, the  {\it type} $tpy(A)$ of an RFS algebra $A$ is defined by $tpy(A):=(\Delta,f,t)$, where $f:=r/m_{\Delta}$ and $t$ is the order of $\sigma$. Here $m_{\Delta}=n,2n-3,11,17$ or $29$ for  $\Delta=A_{n},D_{n},E_{6},E_{7}$ or $E_{8}$,
respectively. Notice that if $n$ is the number of vertices of $\Delta$, then $nf$ is the number of isoclasses of simple $A$-modules.
We remark that $m_{\Delta}$ has the following categorical interpretation (cf. \cite[Section 1.1]{BLR})$\colon$any path of length greater than or equal to $m_{\triangle}$ is zero in the mesh category $k(\mathbb{Z}\Delta)$.

 \begin{Prop}\label{RFS-algebra} {\rm(\cite[Proposition 1.1]{A2})} The set of all types of representation-finite self-injective algebras {\rm($\ncong k$)} is equal to the disjoint union of the following sets.

\begin{enumerate}[$(a)$]
\item $\{(A_{n},s/n,1)\mid n,s\in\mathbb{N}\}$.
\item $\{(A_{2p+1},s,2)\mid p,s\in\mathbb{N}\}$.
\item $\{(D_{n},s,1)\mid n,s\in\mathbb{N},n\geq4\}$.
\item $\{(D_{3m},s/3,1)\mid m,s\in\mathbb{N}, m\geq2,3\nmid s\}$.
\item $\{(D_{n},s,2)\mid n,s\in\mathbb{N},n\geq4\}$.
\item $\{(D_{4},s,3)\mid s\in\mathbb{N}\}$.
\item $\{(E_{n},s,1)\mid n=6,7,8, s\in\mathbb{N}\}$.
\item $\{(E_{6},s,2)\mid s\in\mathbb{N}\}$.
\end{enumerate}
\end{Prop}

Recall from \cite{BG} and \cite{BLR} that a representation-finite $k$-algebra is called {\it standard} if $A$-$\ind\cong k(\Gamma_{A})$, where $k(\Gamma_{A})$ is the mesh category of the Auslander-Reiten quiver $\Gamma_{A}$ of $A$ and $A$-$\ind$ is the full subcategory of $A$-mod whose objects are the representatives of the isoclasses of indecomposable modules. Non-standard algebras are algebras which are not standard.

\begin{Rem} \label{standard-non-standard-pair} \rm(cf. \cite{A1,A2} and \cite[Section 4]{CKL}) Standard RFS algebras appear in all types and non-standard RFS algebras appear only in type $(D_{3m},1/3,1)$ for some $m\geq2$. For every non-standard RFS algebra $A$, there is a standard RFS algebra of the same type, which is denoted by $A_s$ and called the standard counterpart of $A$. The RFS algebras which correspond to symmetric algebras are of types $\{(A_{n},s/n,1)\mid s\in\mathbb{N},s|n\}$, $\{(D_{3m},1/3,1)\}$, $\{(D_{n},1,1)\mid n\in\mathbb{N},n\geq4\}$, $\{(E_{n},1,1)\mid n=6,7,8\}$.
\end{Rem}

Recall from \cite{Riedtmann4,A1} that if $A$ is standard, then we have that $A$-$\stind\cong k(_{s}\Gamma_{A})$, where $k(_{s}\Gamma_{A})$ is the mesh category of the stable Auslander-Reiten quiver $_{s}\Gamma_{A}$ and $A$-$\stind$ is the full subcategory of  $A$-$\stmod$ whose objects are objects in $A$-$\ind$. Moreover, there is a covering functor $F\colon
k(\mathbb{Z}\Delta)\longrightarrow$ $A$-$\stmod$. In particular, for $e,f,h\in \mathbb{Z}\Delta$, there are the following bijections$\colon$
\begin{align}
\bigoplus_{Fh=Ff}\Hom_{k(\mathbb{Z}\Delta)}(e,h)\cong
\Hom_{k(_{s}\Gamma_{A})}(Fe,Ff),\notag
\bigoplus_{Fh=Ff}\Hom_{k(\mathbb{Z}\Delta)}(e,h)\cong \stHom_{A}(Fe,Ff),\notag\\
\bigoplus_{Fe=Fh}\Hom_{k(\mathbb{Z}\Delta)}(e,f)\cong
\Hom_{k(_{s}\Gamma_{A})}(Fh,Ff),\notag \bigoplus_{Fe=Fh}\Hom_{k(\mathbb{Z}\Delta)}(e,f)\cong \stHom_{A}(Fh,Ff) \notag.
\end{align}

In the following two lemmas, we recall the well-known properties on homomorphism spaces in the mesh category of the stable translation quiver $\mathbb{Z}\Delta$, where $\Delta=A_n$ or $\Delta=D_n$. We use the following enumeration on the vertices of $\Delta\colon$
	
	$$\xymatrix{
		1 \ar[r] & 2\ar[r]  & \cdots  \ar[r] & n& (A_{n}),}$$
\ \ \ \ \ \ \ \ \ \ \ \ \ \ \ \ \ \ \ \ \ \ \ \ \ \ \ \ \ \ \ \ \ \ \ \ \ \ \ \ \ \ \ \  \ \ \ \ \ \ \ \ \ \ \ \ \ \ \ \ \ \ \ \ \ \ \ \ \ \ \ \ \ \ \ \ \ \ \ \ \ \ \ \ \ \ \ \ \ \  \ \ \ \ \ \ \ \ \ \ \ \ \ \ \ \ \ \ \ \ \ \ \ \ \ \ \ \ \ \ \ \ \ \ \ \ \ \ \ \ \ \ \ \  (2.1)
	$$\xymatrix{
		& & &n & \\
		1 \ar[r] & 2\ar[r]  & \cdots  \ar[r] & n-2 \ar[u] \ar[r] &n-1 & (D_{n}).
	}$$
It is often convenient to write a vertex of $\mathbb{Z}\Delta$ as its coordinate $(p,q)$, where $p, q$ are integers, $1\leq q\leq n$ and $n$ is the number of vertices of $\Delta$.

 \begin{Lem}\label{Riedtmann-lemma}{\rm(\cite[Lemma 2.6.1]{Riedtmann2})} For any vertices $(p,q)$ and $(r,s)$ in $\mathbb{Z}A_{\ell}$, we have $$dim_{k}(\Hom_{k(\mathbb{Z}A_{\ell})}((p,q),(r,s)))\leq1.$$
In particular, $dim_{k}(\Hom_{k(\mathbb{Z}A_{\ell})}((p,q),(r,s)))=1$ if and only if $p\leq r< p+q\leq r+s\leq p+\ell.$
\end{Lem}
\begin{center}
	\vspace{-2cm}
	\setlength{\unitlength}{0.4cm}
	\begin{picture}(15,15)
	\put(0,0){\line(1,0){20}}
	\put(0,8){\line(1,0){20}}

    \put(7.18,3.67){\line(1,4){0.96}}
    \put(7.78,3.21){\line(1,4){0.96}}
    \put(8.4,2.68){\line(1,4){0.97}}
    \put(9,2.28){\line(1,4){0.97}}
    \put(9.6,1.78){\line(1,4){0.97}}
    \put(10.22,1.34){\line(1,4){0.97}}
    \put(10.78,0.88){\line(1,4){0.97}}
    \put(11.35,0.45){\line(1,4){0.97}}

    \put(7.5,8){\line(4,-3){5.47}}
    \put(6.5,4.1){\line(1,4){0.97}}
	\put(12,0){\line(1,4){0.97}}
    \put(12,0){\line(-4,3){5.47}}
	\put(11.5,-1.1){$(r,1)$} \put(13,4){$(r,s)$} \put(6,8.3){$(r+s-\ell,\ell)$}
	\put(-1.8,4.1){$(r+s-\ell,\ell-s+1)$}
	\end{picture}
\end{center}
\vspace{0.8cm}

Given a category $\mathcal{C}$ and a functor
$F\colon\mathcal{C}\longrightarrow$ $k$-mod, we set Supp$(F):=\{X\in
\mathcal{C} \mid X \mbox{ is indecomposable and }\\ F(X)\neq0\}$.
According to \cite[Section 2]{Riedtmann3}, for the vertices of $\mathbb{Z}D_{n}$, we call a vertex $(p,q)$ {\it low}, if $q\leq n-2$, and {\it high}, otherwise. Notice that these terminologies are still valid for Auslander-Reiten quivers of type $\{(D_{n},s,1)\mid n,s\in\mathbb{N},n\geq4\}$ and type $\{(D_{3m},s/3,1)\mid m,s\in\mathbb{N}, m\geq2,3\nmid s\}$.

 \begin{Lem}\label{support-of-Dn}{\rm(\cite[Proposition 2.1]{Riedtmann3})} Let $(p,q)$ be a vertex of $\mathbb{Z}D_{n}$.

\begin{enumerate}[$(a)$]
\item If $(p,q)$ is low, we have
$\Supp(\Hom_{k(\mathbb{Z}D_{n})}((p,q),-))=\{(x,y)\colon p\leq x\leq p+q-1<x+y\}\cup\{(x,y)\colon x<p+n-1\leq x+min\{y,n-1\}\leq p+q+n-2\}$.
\item If $(p,q)$ is high, we have
 $\Supp(\Hom_{k(\mathbb{Z}D_{n})}((p,q),-))=\{(x,y)\colon y\leq n-2, x\leq p+n-2<x+y\}\cup\{(x,y)\colon y\geq n-1, p\leq x \leq p+n-2, x+y\equiv p+q$ {\rm mod} $2\}$.
\end{enumerate}

 \begin{small}
\begin{center}
\vspace{-0.8cm}
\setlength{\unitlength}{0.4cm}
\begin{picture}(15,14)
 \put(-3,0){\line(1,0){23}}
\put(-3,10){\line(1,0){23}}

\put(2.1,4.76){\line(3,5){3.14}}
\put(2.62,4.06){\line(3,5){3.55}}
\put(3.12,3.16){\line(3,5){4.11}}
\put(3.62,2.36){\line(3,5){4.6}}
\put(4.1,1.6){\line(3,5){5.05}}
\put(4.53,0.86){\line(3,5){5.48}}

\put(8.1,3.5){\line(3,5){3.375}}
\put(8.62,2.72){\line(3,5){3.375}}
\put(9.15,1.83){\line(3,5){3.375}}
\put(9.65,1){\line(3,5){3.375}}

\put(5,0){\line(3,5){6}}
\put(5,0){\line(-3,5){3.41}}
\put(4.2,10){\line(-3,-5){2.58}}
\put(4.2,10){\line(3,-5){6}}
\put(10.15,0){\line(3,5){3.45}}
\put(11,10){\line(3,-5){2.57}}
\put(13.6,5.76){$(p+n-2,q)$}
 \put(9,10.3){$(p+q-1,n-1)$}\put(2,10.3){$(p,n-1)$}
\put(1.5,-0.8){$(p+q-1,1)$} \put(9.5,-0.8){$(p+n-2,1)$}
\put(-0.4,5.4){$(p,q)$}
\end{picture}
\end{center}
\end{small}
\begin{small}
\begin{center}
\vspace{-0.6cm}
\setlength{\unitlength}{0.35cm}
\begin{picture}(15,14)
 \put(-3,0){\line(1,0){23}}
\put(-3,10){\line(1,0){23}}
\put(3.82,8.6){\line(3,5){0.856}}
\put(4.36,7.82){\line(3,5){1.32}}
\put(4.89,6.94){\line(3,5){1.8}}
\put(5.4,6.10){\line(3,5){2.33}}
\put(5.86,5.20){\line(3,5){2.87}}
\put(6.35,4.44){\line(3,5){3.32}}
\put(6.86,3.66){\line(3,5){3.78}}
\put(7.34,2.8){\line(3,5){4.3}}
\put(7.9,1.82){\line(3,5){4.9}}
\put(8.45,1){\line(3,5){5.4}}
\put(9,0){\line(3,5){6}}
\put(9,0){\line(-3,5){6}}
\put(14.4,9){\line(1,0){1}}
\put(13,10.3){$(p+n-2 ,n-1)$} \put(0.8,10.3){$(p,n-1)$}
\put(6.5,-0.8){$(p+n-2,1)$} \put(15.8,8.8){$(p+n-2,n)$}
\end{picture}
\end{center}
\end{small}
\end{Lem}

\begin{Rem}
\begin{enumerate}[$(1)$]
\item According to \cite[Proposition 2.1]{Riedtmann3}, we have the following.

$$\Supp(\Hom_{k(\mathbb{Z}D_{n})}(-,(p+n-2,q)))=\Supp(\Hom_{k(\mathbb{Z}D_{n})}((p,q),-)),$$
$$\Supp(\Hom_{k(\mathbb{Z}D_{n})}(-,(p+n-2 ,n-1)))=\Supp(\Hom_{k(\mathbb{Z}D_{n})}((p,n-1),-))\mbox{ for }n\mbox{ even},$$
$$\Supp(\Hom_{k(\mathbb{Z}D_{n})}(-,(p+n-2 ,n)))=\Supp(\Hom_{k(\mathbb{Z}D_{n})}((p,n-1),-))\mbox{ for }n\mbox{ odd}.$$
\item The second part of the union in Lemma \ref{support-of-Dn} $(b)$ means that if $(p,q)$ is  high, then high vertex $(x,y)$ is in $\Supp(\Hom_{k(\mathbb{Z}D_{n})}((p,q),-))$ when  $x+y$ and $p+q$ have the same parity.
\end{enumerate}
\end{Rem}

We remark that the mesh category of $\mathbb{Z}A_n$ (resp. $\mathbb{Z}D_n$) can be identified with $D^b(kA_n)$-ind  (resp. $D^b(kD_n)$-ind), where $D^b(kA_n)$ (resp. $D^b(kD_n)$) denotes the bounded derived category of the path algebra $kA_n$ (resp. $kD_n$). From this point of view, Lemma \ref{Riedtmann-lemma} and Lemma \ref{support-of-Dn} describe the homomorphism spaces in $D^b(kA_n)$ and in $D^b(kD_n)$ respectively, and Remark 2.10 (1) is an explicit description of Serre duality in $D^b(kD_n)$ {\rm(cf. \cite{H})}.

\subsection{Simple-minded system}

\begin{Def}\label{definition-sms-representation-finite} {\rm(cf. \cite[Definition 2.4, 2.5]{Dugas})} Let $A$ be a self-injective algebra over an algebraically closed field $k$.
A family of objects $\mathcal{S}$ in $A$-$\stmod$ is a simple-minded system {\rm(sms for short)} if the following two conditions are satisfied$\colon$
\begin{enumerate}[$(1)$]
\item {\rm(Orthogonality)} For any $S,T\in\mathcal{S}$,
$\stHom_A(S,T)\cong \left\{\begin{array}{ll} 0 & (S\neq T), \\
k & (S=T).\end{array}\right.$
\item {\rm(Generating condition)} $\mathcal{F}(\mathcal{S})=A$-$\stmod$.
\end{enumerate}
\end{Def}

\begin{Rem}\label{reformulation-sms} Let $A$ be an RFS algebra and $\mathcal{S}$ an orthogonal system in $A$-$\stmod$. Then, according to \cite[Theorem 5.6]{KL}, the generating condition in an sms can be replaced by the following weak-generating condition$\colon$for any indecomposable non-projective
$A$-module $X$, there exists some $S\in \mathcal{S}$ such that $\stHom_A(X,S)\\\neq 0$.  Indeed, this fact gives the direct connection between sms's and combinatorial configurations at least for standard RFS algebras {\rm(cf. \cite{CKL})}.
\end{Rem}

Recall from \cite[Section 2, Section 4]{CKL} that for an RFS algebra $A$, if $\mathcal{S}$ is an sms in $A$-$\stmod$, then $\mathcal{S}$ satisfies the following three conditions.

\begin{enumerate}[$(1)$]
\item $\mathcal{S}$ is an orthogonal system in $A$-$\stmod$.
\item The cardinality of the set $\mathcal{S}$ is equal to the number of non-isomorphic simple $A$-modules.
\item $\mathcal{S}$ is Nakayama-stable, that is, the Nakayama functor $\nu$ permutes the elements of $\mathcal{S}$.
\end{enumerate}

Notice that the above condition (1) is obvious, but (2) and (3) are highly nontrivial. In fact, they are consequences of the following Liftability theorem (cf. \cite{CKL} Theorem 4.1)$\colon$if $\mathcal{S}$ is an sms over  RFS algebra $A$, then there is another RFS algebra $B$ and a derived equivalence $F\colon\mathcal {D}^{b}(B)\rightarrow \mathcal {D}^{b}(A)$ such that the induced stable equivalence $\widetilde{F}\colon B$-$\stmod\rightarrow A$-$\stmod$ maps simple $B$-modules into $\mathcal{S}$.

As a comparison, we would like to mention two interesting facts on combinatorial configurations for a general self-injective algebra$A\colon$any combinatorial configuration is Nakayama-stable (cf. \cite[Theorem 6.2]{J1}); a combinatorial configuration $\mathcal{C}$ is a simple-minded system if and only if $\mathcal{F}(\mathcal{C})$ is functorially finite in $A$-$\stmod$ (cf. \cite[Proposition 2.13]{CP}).

\section{A new characterization  for an orthogonal system to be an sms}

\subsection{Main result and its proof}
In this subsection, we show that the three conditions (1), (2) and (3) in last subsection are sufficient for $\mathcal{S}$ to be an sms. That is, we prove the following theorem.

\begin{Thm}\label{sms-sufficient-condition} Let $A$ be an RFS algebra and $\mathcal{S}$ a family of objects in $A$-$\stmod$. Then $\mathcal{S}$ is an sms if and only if $\mathcal{S}$ satisfies the following three conditions.
\begin{enumerate}[$(1)$]
\item $\mathcal{S}$ is an orthogonal system in $A$-$\stmod$.
\item The cardinality of $\mathcal{S}$ is equal to the number of non-isomorphic simple $A$-modules.
\item $\mathcal{S}$ is Nakayama-stable, that is, the Nakayama functor $\nu$ permutes the objects of $\mathcal{S}$.
\end{enumerate}
\end{Thm}

The proof of Theorem \ref{sms-sufficient-condition} will be given after we prove a technical lemma on orthogonality in $A$-$\stmod$ (Lemma \ref{two-side-orthogonal}).

\begin{Rem}\label{counter-examples}
\begin{enumerate}[$(a)$]
\item We cannot delete any condition from $(1),(2),(3)$. A counterexample for deleting $(3)$ comes from the self-injective Nakayama algebra $A$, where $A=kQ/I$ is given by the following quiver $Q$
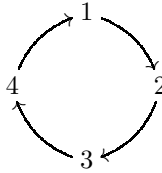
\vspace{0.1cm}
$$\xymatrix@rd{
1\ar@/^/[r]&2\ar@/^/[d]\\
4\ar@/^/[u]&3\ar@/^/[l] }$$
\vspace{-0.45cm}

\noindent with admissible ideal $I=rad^{4}(kQ)$.
It is easy to check that
$\mathcal{S}=\left\{1,~~\begin{matrix}2\\3\\4\end{matrix}, ~~3, ~~\begin{matrix}4\\1\\2\end{matrix} \right\}$ satisfies $(1),(2)$ but not $(3)$, and $\mathcal{S}$ is not an sms.

\item In general, Theorem \ref{sms-sufficient-condition} does not hold for representation-infinite self-injective algebras. A counterexample is given by the algebra $k[x,y]/(x^{2},y^{2})$. It is a $4$-dimensional symmetric local algebra and its Auslander-Reiten quiver consists of a component containing the simple module and a $\mathbb{P}_{1}(k)$-family of  homogenous tubes. We take a module $X$ on the mouth of a homogenous tube, notice that any homogenous tube has only one indecomposable module on its mouth. Let $\mathcal{S}=\{X\}$. It is easy to check that  $\mathcal{S}$ satisfies the above three conditions. However, $\mathcal{S}$ is not an sms since $\mathcal{F}(\mathcal{S})$ is the additive closure of all the modules in  the homogenous tube which contains $X$.
\item It would be interesting to know whether Theorem \ref{sms-sufficient-condition} is false for every representation-infinite algebra.
 \end{enumerate}
\end{Rem}

The main tools in proving Theorem \ref{sms-sufficient-condition} are torsion pair theory and covering theory. One result we need from torsion pair theory is the following lemma, which is a special case of \cite[Proposition 2.3 (1)]{IY}.

\begin{Lem}\label{two-torsion-pairs} Let $A$ be a self-injective algebra and $\mathcal{S}$ an orthogonal system in $A$-$\stmod$. If the subcategory $\mathcal{F}(\mathcal{S})$ is functorially finite in $A$-$\stmod$, then both $(\mathcal{^\perp S}, \mathcal{F}(\mathcal{S}))$ and $(\mathcal{F}(\mathcal{S}),\mathcal{S^{\perp}})$ are torsion pairs in $A$-$\stmod$.
\end{Lem}

\begin{Rem}
 The condition that $\mathcal{F}(\mathcal{S})$ is functorially finite  in  Hom-finite, Krull-Schmidt triangulated $k$-categories is very useful  and  applied  in a number of recent works  {\rm(cf. \cite{CP,CPP,J2})}. The condition that $\mathcal{F}(\mathcal{S})$ is functorially finite in $A$-$\stmod$ clearly holds for RFS algebras.  It would also be interesting to find applications of Lemma \ref{two-torsion-pairs} for representation-infinite algebras. 
\end{Rem}

We now prove three lemmas (Lemma \ref{Dn-En-stable-brick} to Lemma \ref{orthogonal-nu-orbit-of-D_{3m}}) on orthogonality properties in the stable categories of  several classes of RFS algebras, based on the description of supports in mesh categories of $\mathbb{Z}A_n, \mathbb{Z}D_n $  and the covering theory in \cite{BLR}, \cite{Riedtmann2}, \cite{Riedtmann3} and \cite{Riedtmann4}.  We shall freely switch between nonzero indecomposable modules  over RFS algebras and vertices in the corresponding mesh category $k(\mathbb{Z}\Delta)$, where $\Delta$ is a Dynkin quiver.
	

 \begin{Lem}\label{Dn-En-stable-brick} Let $A$ be an RFS algebra of type $(A_{2p+1},s,2)$,
  $D_{n}$ or $E_{n}$ {\rm(except type $(D_{3m},1/3,1)$ for some $m\geq2$).} Then every indecomposable module $X$ is a stable brick in $A$-$\stmod$.
\end{Lem}
\begin{proof}
Recall that under our assumptions all algebras are standard and we can identify $A$-$\stind$ with $k(_{s}\Gamma_{A})\cong k(\mathbb{Z}\Delta)/\langle\sigma\tau^{-m_{\Delta}f}\rangle$. Let X be an indecomposable module  in $A$-$\stmod$ and  $G$ the infinite  cyclic group generated by $\sigma\tau^{-m_{\Delta}f}$. From covering theory we know that
\vspace{0.2cm}
$$\stHom_{A}(X,X)=\Hom_{k(_{s}\Gamma_{A})}(X,X)\cong\bigoplus_{g(E)=X, g\in G}\Hom_{k(\mathbb{Z}\Delta)}(E, Y).$$
Let $\ell(X,X)$ be the minimal length of all the nontrivial paths from $E$ to $X$ in $k(\mathbb{Z}\Delta)$, where $E$ varies in $\mathbb{Z}\Delta$ and satisfies $g(E)=X$ for some $g\in G$. For a given vertex $X$, in the proof below we will see that the vertex $E$ corresponding to the minimal length is unique. We claim that $\ell(X,X)$ is greater than $m_{\Delta}$ and therefore $\Hom_{k(\mathbb{Z}\Delta_{n})}(E, X)=0$.

We first consider the case $\{(D_{3m},s/3,1)\mid m,s\in \mathbb{N}, m,s\geq2,3\nmid s\}$. In this case $f=s/3$ and $k(_{s}\Gamma_{A})\cong k(\mathbb{Z}D_{3m})/\langle\tau^{-s(m_{D_{3m}})/3}\rangle=k(\mathbb{Z}D_{3m})/\langle\tau^{-s(2m-1)}\rangle$.
It follows that any $E$ with $g(E)=X$ for some $g\in G$ has the form $\tau^{-sz(2m-1)}(X)$ for some  integer $z$.
Now it is easy to see that $\ell(X,X)$ is $2s(2m-1)$ and it is greater than $m_{D_{3m}}=2\times (3m)-3$, where $s\geq2$.

For the other cases, $f$ is always a positive integer and $k(_{s}\Gamma_{A})\cong k(\mathbb{Z}\Delta)/\langle\sigma\tau^{-m_{\Delta}f}\rangle$. By \cite[Proposition 2.1]{A2}, for type $A_{n}$, $D_{n}$ or $E_{n}$, the automorphism $\sigma$ of $\mathbb{Z}\Delta$ is induced from some automorphism of $\overrightarrow{\Delta}$ (under the choice of orientation on $\Delta$ given in \cite[Section 2]{A2}). For the convenience of the reader, we list the orientation on $\Delta$ used in \cite[Section 2]{A2} according to the type of $\Delta$.

\vspace{0.15cm}
$$\xymatrix{
    1  & 2 \ar[l] & \cdots \ar[l] & p \ar[l] &\ar[l] p+1\ar[r]& p+2 \ar[r] &\cdots\ar[r]& n\ \  (\overrightarrow{A_{n}}, n=2p+1, p\in \mathbb{Z}), }$$

$$\xymatrix{
     & &  & n  &  \\
   1  & 2 \ar[l] & \cdots \ar[l] & n-2 \ar[l]\ar[r] \ar[u] & n-1\ \  (\overrightarrow{D_{n}}, n\geq4), }$$


$$\xymatrix{
		& & n & &  \\
		1  & 2 \ar[l] & 3\ar[r]\ar[l]\ar[u] &\cdots\ar[r] &n-1\ \    (\overrightarrow{E_{n}}, n=6, 7, 8).}$$

 \vspace{0.1cm}

\noindent We define a $\overrightarrow{\Delta}$-line to be a set of vertices of the form $\tau^z(\overrightarrow{\Delta})$ in $\mathbb{Z}\Delta$ for some $z\in \mathbb{Z}$.
 Therefore $\sigma\tau^{-m_{\Delta}f}(E)$ is in the same $\overrightarrow{\Delta}$-line with $\tau^{-m_{\Delta}f}(E)$ in $k(\mathbb{Z}\Delta)$. For a given $X$ in a $\overrightarrow{\Delta}$-line $\overrightarrow{\Delta}$, the unique vertex $E$ corresponding to the minimal length lies in $\overrightarrow{\Delta}$-line $\tau^{m_{\Delta}f}(\overrightarrow{\Delta})$.
It is easy to check case by case that the length of any path from a vertex in $\tau^{m_{\Delta}f}(\overrightarrow{\Delta})$ to $X$ is greater than or equal to $2fm_{\Delta}-(n-1)$, which is again greater than $m_{\Delta}$.
We illustrate the result with the case $\{(A_{2p+1},s,2)\mid p,s\in\mathbb{N}\}$ in the picture below.
	\begin{small}
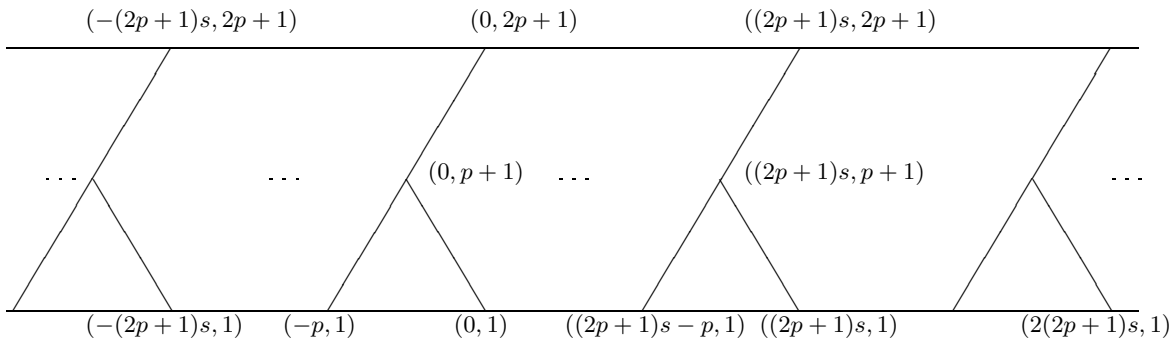
		\begin{center}
			\vspace{-0.5cm}
			\setlength{\unitlength}{0.35cm}
			\begin{picture}(15,14)
			\put(-14,0){\line(1,0){43}}
			\put(-14,10){\line(1,0){43}}
			\put(28,0){\line(-3,5){3}}
			\put(21.95,0){\line(3,5){6}}
			\put(16.1,0){\line(-3,5){3}}
			\put(4.2,10){\line(-3,-5){6}}
			\put(10.15,0){\line(3,5){6}}
			\put(4.2,0){\line(-3,5){3}}
			\put(-7.7,0){\line(-3,5){3}}
			\put(-13.75,0){\line(3,5){6}}
			
			\put(-11.5,5){\line(1,0){0.1}}
			\put(-12,5){\line(1,0){0.1}}
			\put(-12.5,5){\line(1,0){0.1}}	
			\put(-4,5){\line(1,0){0.1}}
			\put(-3.5,5){\line(1,0){0.1}}
			\put(-3,5){\line(1,0){0.1}}
			\put(7,5){\line(1,0){0.1}}
			\put(7.5,5){\line(1,0){0.1}}
			\put(8,5){\line(1,0){0.1}}
			\put(28,5){\line(1,0){0.1}}
			\put(28.5,5){\line(1,0){0.1}}
			\put(29,5){\line(1,0){0.1}}
			
			\put(-3.5,-0.8){$(-p,1)$}
			\put(3,-0.8){$(0,1)$}
			\put(3.6,10.8){$(0,2p+1)$}
			\put(-11,-0.8){$(-(2p+1)s,1)$}
			\put(-11,10.8){$(-(2p+1)s,2p+1)$}
			\put(14.6,-0.8){$((2p+1)s,1)$}
			\put(24.5,-0.8){$(2(2p+1)s,1)$}
			\put(7.2,-0.8){$((2p+1)s-p,1)$}
			\put(14,5){$((2p+1)s,p+1)$}
			\put(2,5){$(0,p+1)$}
			\put(14,10.8){$((2p+1)s,2p+1)$}
			\end{picture}
		\end{center}
	\end{small}
	\vspace{0.5cm} 
	We only need to show that $\stHom_{A}((0,i),(0,i))\cong k$ for vertex $(0,i)$ for  each $1\leq i \leq 2p+1$,  since any other vertex can be obtained by the power of self-equivalence $\tau$ of $A$-$\stmod$. 
	If we assume that $(0,i)$  is in $\overrightarrow{\Delta}$-line for $1\leq i\leq2p+1$, then $\sigma\tau^{m_{\Delta}f}((0,i))=(-(2p+1)s,2p+2-i)$ and the vertex in $\sigma\tau^{m_{\Delta}f}(\overrightarrow{\Delta}$)-line is of the form $(-(2p+1)s,j)$ for $1\leq j\leq 2p+1$.  It is easy to see that
	$\ell((0,i),(0,i))$ is the length of path from $(-(2p+1)s,2p+2-i)$ to $(0,i)$ in  $k(\mathbb{Z}A_{2p+1})$, it is greater than or equal to the length of path from $(-(2p+1)s,2p+2-i)$  to $(0,p+1)$  in  $k(\mathbb{Z}A_{2p+1})$,  which is $2(2p+1)s-|p+1-i|$. Since $2(2p+1)s-|p+1-i|\geq 2(2p+1)s-2p > 2p+1$, $\ell((0,i),(0,i))>2p+1$, where $|a|$ is the absolute value of a number $a$.

Summarizing the above discussion, in all cases we have $\stHom_{A}(X,X)\cong \Hom_{k(\mathbb{Z}\Delta)}(X, X)\cong k$. 
\end{proof}

 \begin{Lem}\label{orthogonal-nu-orbit} Let $A$ be an  RFS algebra of type $(A_{2p+1},s,2)$, $D_{n}$ or $E_{n}$ {\rm(except  type $\{(D_{3m},s/3,1)\mid m,s\in\mathbb{N}, m\geq2,3\nmid s\}$).} Then the $\nu$-orbit $O_{\nu}(X)$ of any indecomposable module $X$ is an orthogonal system in $A$-$\stmod$.
\end{Lem}
\begin{proof}
We first note that $k(_{s}\Gamma_{A})\cong k(\mathbb{Z}\Delta)/\langle\sigma\tau^{-m_{\Delta}f}\rangle$. Let $X$, $Y$ be  indecomposable modules in $A$-$\stmod$ and $G$ the infinite  cyclic group generated by $\sigma\tau^{-m_{\Delta}f}$.
From covering theory, we have
\vspace{0.2cm}
$$\stHom_{A}(X,Y)=\Hom_{k(_{s}\Gamma_{A})}(X,Y)\cong\bigoplus_{g(E)=X, g\in G}\Hom_{k(\mathbb{Z}\Delta)}(E, Y).$$
Let $\ell(X,Y)$ be the minimal length of all the nontrivial paths from $E$ to $Y$ in $k(\mathbb{Z}\Delta)$,
where $E$ varies in $\mathbb{Z}\Delta$ and satisfies $g(E)=X$ for some $g\in G$. According to the result in \cite[Proposition 1.5 and 1.6]{BS}, we know that the Nakayama functor $\nu\cong\tau^{-m_{\Delta}}$ in $k(\mathbb{Z}\Delta)$ . It follows that $\nu(X)\cong\tau^{-m_{\Delta}}(X)$ in $k(\mathbb{Z}\Delta)/\langle\sigma\tau^{-m_{\Delta}f}\rangle$. Any $E$ with $g(E)=\nu(X)$ has the form $\sigma^{z}\tau^{-m_{\Delta}(zf+1)}(X)$ for some integer $z$, notice that $\sigma\tau=\tau\sigma$ in $k(\mathbb{Z}\Delta)$.

Under the above notations, we now show that  the $\nu$-orbit $O_{\nu}(X)$ of $X$ is an orthogonal system in $A$-$\stmod$.
There are two cases to be considered. For the case of symmetric algebras, we know that $\nu\cong id$, by Lemma \ref{Dn-En-stable-brick}, the $\nu$-orbit $O_{\nu}(X)$ of $X$ is an orthogonal system.
For the other cases, we show that $\stHom_{A}(\nu^{i}(X),\nu^{j}(X))=0$ for integers $i,j,$ where $i\neq j,0\leq i,j<m$ and $m$ is the order of $\nu$ for $X$.
It is easy to see that $\ell({\nu^{i}(X)},\nu^{j}(X))=2(\overline{j-i})m_{\Delta}$ for $i\neq j$,
which is  greater than $m_{\Delta}$, where $ j-i\equiv\overline{j-i}$ mod $n$. Therefore
$\stHom_{A}(\nu^{i}(X),\nu^{j}(X))=0.$  It follows that the $\nu$-orbit $O_{\nu}(X)$ of any indecomposable module $X$ is an orthogonal system in $A$-$\stmod$.

We illustrate  the above result through the case $\{(D_{n},s,2)\mid n,s\in\mathbb{N}, s\geq2, n>4 \}$ in the picture below.
	
	\begin{small}
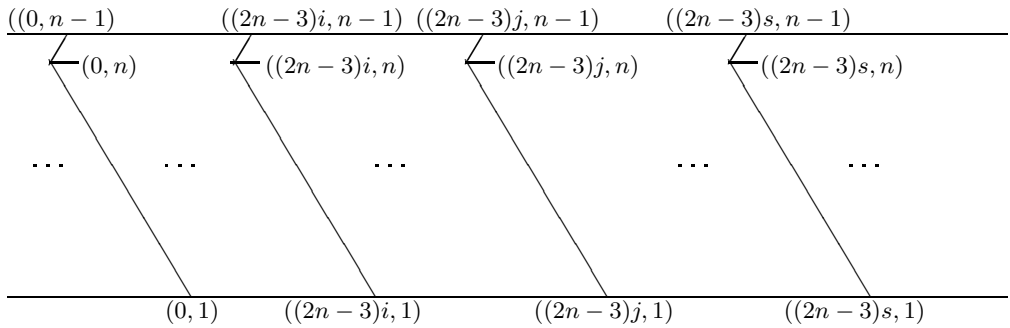
		\begin{center}
			\vspace{-0.6cm}
			\setlength{\unitlength}{0.35cm}
			\begin{picture}(15,14)
			\put(-12,0){\line(1,0){38}}
			\put(-12,10){\line(1,0){38}}
			
			\put(15.4,8.9){\line(1,0){1.1}}
			\put(-3.5,8.9){\line(1,0){1.1}}
			\put(5.4,8.9){\line(1,0){1.1}}
			
			\put(-5,0){\line(-3,5){5.4}}
			\put(-9.7,10){\line(-3,-5){.7}}
			\put(-10.4,8.9){\line(1,0){1.1}}
			
			\put(2,0){\line(-3,5){5.4}}
			\put(-2.7,10){\line(-3,-5){.7}}
			
			\put(10.8,0){\line(-3,5){5.4}}
			\put(6.1,10){\line(-3,-5){.7}}
			
			\put(20.8,0){\line(-3,5){5.4}}
			\put(16.1,10){\line(-3,-5){.7}}
			
			\put(20,5){\line(1,0){0.1}}
			\put(20.5,5){\line(1,0){0.1}}
			\put(21,5){\line(1,0){0.1}}
			\put(13.5,5){\line(1,0){0.1}}
			\put(14,5){\line(1,0){0.1}}
			\put(14.5,5){\line(1,0){0.1}}	
			\put(2,5){\line(1,0){0.1}}
			\put(2.5,5){\line(1,0){0.1}}
			\put(3,5){\line(1,0){0.1}}	
			\put(-6,5){\line(1,0){0.1}}
			\put(-5.5,5){\line(1,0){0.1}}
			\put(-5,5){\line(1,0){0.1}}	
			\put(-10,5){\line(1,0){0.1}}
			\put(-10.5,5){\line(1,0){0.1}}
			\put(-11,5){\line(1,0){0.1}}
			
			\put(-12,10.3){$((0 ,n-1)$}
			\put(-9.2,8.5){$(0,n)$} 
			\put(-6,-0.8){$(0,1)$}	
			\put(-3.9,10.3){$((2n-3)i,n-1)$}
			\put(-2.2,8.5){$((2n-3)i,n)$}
			\put(-1.5,-0.8){$((2n-3)i,1)$}
			\put(3.5,10.3){$((2n-3)j ,n-1)$} 
			\put(6.6,8.5){$((2n-3)j,n)$}
			\put(8,-0.8){$((2n-3)j,1)$}
			\put(13,10.3){$((2n-3)s ,n-1)$} 
			\put(16.6,8.5){$((2n-3)s,n)$}
			\put(17.5,-0.8){$((2n-3)s,1)$} 
			
			\end{picture}
		\end{center}
	\end{small}
	\vspace{0.6cm}
	We only need  to show that the $\nu$-orbit $O_{\nu}((0,p))$ of any indecomposable module $(0,p)$ is an orthogonal system in $A$-$\stmod$ for  all $1\leq p \leq n$, since any other vertex can be obtained by the power of self-equivalence $\tau$ of $A$-$\stmod$. If $1 \leq p\leq n-2$, then $\sigma\tau^{-m_{D_{n}}f}(0,p)=((2n-3)s,p)$, $\sigma\tau^{-m_{D_{n}}f}(0,n)=((2n-3)s,n-1)$ and $\sigma\tau^{-m_{D_{n}}f}(0,n-1)=((2n-3)s,n).$
	It is also clear that $\nu^{q}((0,p))=((2n-3)q,p)$ for $1\leq p \leq n$ and $ q\in\mathbb{Z}$,  $\sigma\nu^{s}((0,n))=((2n-3)s,n-1)$ and $\sigma\nu^{s}((0,n-1))=((2n-3)s,n)$.
	We can see that $\ell({\nu^{i}((0,p))},\nu^{j}((0,p)))=2(\overline{j-i})(2n-3)$ for $i\neq j$,  $0\leq i,j<s$ and $1\leq p\leq n,$ which is greater than $2n-3$. Therefore $\stHom_{A}(\nu^{i}((0,p)),\nu^{j}(0,p)))=0,$  it follows that the $\nu$-orbit $O_{\nu}((0,p))$  is an orthogonal system in $A$-$\stmod$.
\end{proof}

\begin{Lem}\label{orthogonal-nu-orbit-of-D_{3m}} Let $A$ be a standard RFS algebra of type $\{(D_{3m},s/3,1)\mid m,s\in\mathbb{N}, m\geq2,3\nmid s\}$ and let $X=(p,q)$ be a vertex in $_{s}\Gamma_{A}$ for some integers $p,q$, where $1\leq q\leq3m$. Then we have the following.
\begin{enumerate}[$(1)$]
\item If $1\leq q<m$ or $q\geq3m-1$, then $X$ is a stable brick.
\item If  $1\leq q<m$ or $q\geq3m-1$, then the $\nu$-orbit $O_{\nu}(X)$ of $X$ is an orthogonal system.
\item If $m\leq q<3m-1$, then the $\nu$-orbit $O_{\nu}(X)$ of $X$ is not an orthogonal system.
\end{enumerate}
\end{Lem}
\begin{proof}
Let  $G$ be the infinite  cyclic group generated by $\tau^{-s(2m-1)}$. Notice that $k(_{s}\Gamma_{A})\cong k(\mathbb{Z}D_{3m})/  \langle\tau^{-s(2m-1)}\rangle$.
 Let $Y$ be a vertex in $_{s}\Gamma_{A}$. From covering theory, we have
 \vspace{0.2cm}
$$\stHom_{A}(X,Y)=\Hom_{k(_{s}\Gamma_{A})}(X,Y)\cong\bigoplus_{g(E)=X, g\in G}\Hom_{k(\mathbb{Z}D_{3m})}(E, Y).$$
Let $\ell(X,Y)$ be the minimal length of all the nontrivial paths from $E$ to $Y$ in $k(\mathbb{Z}D_{3m})$,
where $E$ varies in $\mathbb{Z}D_{3m}$ and satisfies $g(E)=X$ for some $g\in G$.

(1) There are two cases to be considered. For the case $\{(D_{3m},s/3,1)\mid m,s\in \mathbb{N}, m,s\geq2,3\nmid s\}$, by Lemma \ref{Dn-En-stable-brick},  $X$ is a stable brick in $_{s}\Gamma_{A}$ for all $1\leq q \leq 3m$.

For the other case $\{(D_{3m},1/3,1)\}$ for some $m\geq2$.
 We first assume that $q\geq3m-1$, that is, $X$ is a high vertex, which is denoted by $(p,3m-1)$ (resp. $(p,3m)$) for some integer $p$. It is sufficient to show  that except the trivial path from $X$ to $X$,  any path from the  above $E$ to $X$ is zero in $k(\mathbb{Z}D_{3m})$.
Since $s=1$, we know that the above  $E$ is of the form $\tau^{z(2m-1)}(X)$ for some integer $z$. There are three subcases to be considered.

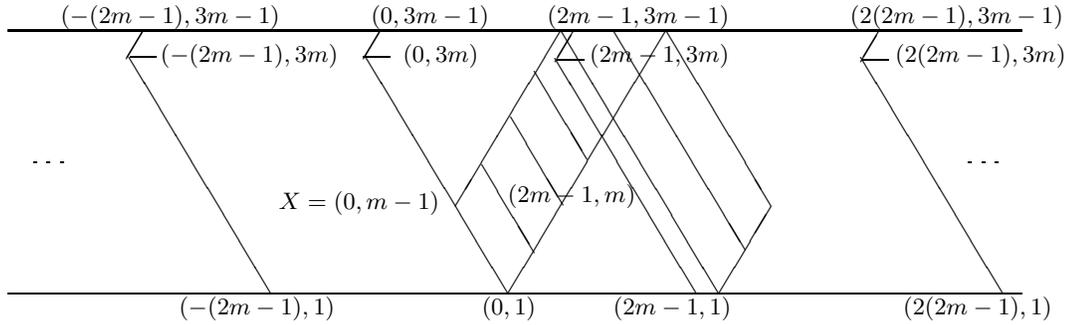
\begin{small}
	\begin{center}
		\vspace{-0.6cm}
		\setlength{\unitlength}{0.35cm}
		\begin{picture}(15,14)
		\put(-10,0){\line(1,0){38.5}}
		\put(-10,10){\line(1,0){38.5}}
		
		\put(-9,5){\line(1,0){0.1}}
		\put(-8.5,5){\line(1,0){0.1}}
		\put(-8,5){\line(1,0){0.1}}
		
		\put(26.5,5){\line(1,0){0.1}}
		\put(27,5){\line(1,0){0.1}}
		\put(27.5,5){\line(1,0){0.1}}
		
		\put(16.15,0){\line(-3,5){5.37}}
		\put(11.5,10){\line(-3,-5){.7}}
		\put(10.88,8.9){\line(1,0){1}}
		
	   \put(-5.35,9){\line(1,0){1}}
		\put(0,0){\line(-3,5){5.4}}
		\put(-4.85,10){\line(-3,-5){.65}}
		
		 \put(3.5,9){\line(1,0){1}}
		\put(9,0){\line(-3,5){5.4}}
		\put(4.15,10){\line(-3,-5){.65}}
		
	\put(11,10){\line(-3,-5){4}}
		\put(9,0){\line(3,5){6}}
		\put(11,10){\line(3,-5){6}}
		\put(17,0){\line(3,5){2}}
		\put(15,10){\line(3,-5){4}}
		\put(8,4.9){\line(3,-5){2}}
		
		\put(9.1,6.7){\line(3,-5){2}}
		\put(10,8.45){\line(3,-5){2}}
		\put(13,10){\line(3,-5){5}}
		
		\put(27.8,0){\line(-3,5){5.4}}
		\put(22.45,8.9){\line(1,0){1}}
		\put(23.1,10){\line(-3,-5){.7}}
		
		\put(3.8,10.3){$(0,3m-1)$}
		\put(5,8.8){$(0,3m)$}
		\put(8,-0.8){$(0,1)$}
		\put(.3,3.2){$X=(0,m-1)$}
		\put(9,3.5){$(2m-1,m)$}

		\put(10.5,10.3){$(2m-1,3m-1)$}
		 \put(12.1,8.8){$(2m-1,3m)$}
		 \put(13,-0.8){$(2m-1,1)$}
		
		\put(23.7,8.8){$(2(2m-1),3m)$}
	  \put(22,10.3){$(2(2m-1),3m-1)$}
		\put(24,-0.8){$(2(2m-1),1)$}
				
		 \put(-8,10.3){$(-(2m-1),3m-1)$}
		\put(-4.2,8.8){$(-(2m-1),3m)$}
		\put(-3.5,-0.8){$(-(2m-1),1)$}
		\end{picture}
	\end{center}
\end{small}
\vspace{.6cm}
\begin{enumerate}[(i)]
\item If $z$ is a negative integer, it is clear that $\Hom_{k(\mathbb{Z}D_{3m})}(E, X)=0$.
\item If $z\geq2$, then $\ell(X,X)$ is greater than or equal to $4(2m-1)$, which is greater than  $m_{D_{3m}}=2\times (3m)-3$. It follows that $\Hom_{k(\mathbb{Z}D_{3m})}(E, X)=0$.
\item If $z=1$, then $E=(-2m+1+p,3m-1)$ (resp. $(-2m+1+p,3m)$). Since $p$ and $-2m+1+p$ do not have the same parity and by the description of support of a high vertex, $\Hom_{k(\mathbb{Z}D_{3m})}(E, X)=0$.
\end{enumerate}
 Then we assume that $1\leq q<m$. From the description of the support of a low vertex, it is easy to see that  $\tau^{z(2m-1)}(X)$ is not in Supp$(\Hom_{k(\mathbb{Z}D_{3m})}(-,(p,q)))$ for any nonzero integer $z$. It follows that $\stHom_{A}(X,X)\cong \Hom_{k(\mathbb{Z}D_{3m})}(X, X)\cong k.$

(2) According to the result in \cite[Proposition 1.5 and 1.6]{BS}, we know  that the Nakayama functor $\nu\cong\tau^{-m_{\Delta}}$ in $k(\mathbb{Z}\Delta)$, where $\Delta$ means Dynkin quiver $A_{n}$, $D_{n}$ or $E_{n}$.
Since $k(_{s}\Gamma_{A})\cong k(\mathbb{Z}D_{3m})/\langle\tau^{-s(2m-1)}\rangle$, $\nu(X)\cong\tau^{-m_{D_{3m}}}(X)\cong\tau^{(zs-3) m_{D_{3m}}/3}(X)=\tau^{(zs-3)(2m-1)}(X)$ in $k(\mathbb{Z}D_{3m})/\langle\tau^{-s(2m-1)}\rangle$ for all integers $z$.
 Moreover, since $3\nmid s$, there is a smallest positive integer $e$ such that $\nu^{e}\cong\tau^{-(2m-1)}$ in $k(\mathbb{Z}D_{3m})/\langle\tau^{-s(2m-1)}\rangle$.
 It follows that $O_{\nu}(X)=\{X, \nu^{e}(X),\cdots,\nu^{(s-1)e}(X)\}=\{X,\tau^{-(2m-1)}(X),\cdots,\tau^{-(s-1)(2m-1)}(X)\}$.

There are two cases as follows. For the case $\{(D_{3m},1/3,1)\}$ for some $m\geq2$, $A$ is a symmetric algebra and $\nu\cong id$. Since $X$ is stable brick for $1\leq q<m$ and $q\geq3m-1$, the $\nu$-orbit $O_{\nu}(X)$ of $X$ is an orthogonal system in $_{s}\Gamma_{A}$.

For the other case $\{(D_{3m},s/3,1)\mid m,s\in \mathbb{N}, m,s\geq2,3\nmid s\}$.
We show that  $\stHom_{A}(\nu^{i}(X),\nu^{j}(X))=0$ for any $i\neq j$, where $0\leq i,j<s$. There are two subcases to be considered.
\begin{enumerate}[(i)]
\item Let $\mathcal{R}_{1}:=\{(i,j)|i\neq j,\ell({\nu^{i}(X)},\nu^{j}(X))=2(2m-1)\}$. Since $\nu^{e}\cong\tau^{-(2m-1)}$,  $\mathcal{R}_{1}$ is not an empty set. Similarly to the proof of (1) of the case  $\{(D_{3m},1/3,1)\}$ for some $m\geq2$, we know that $\stHom_{A}(\nu^{i}(X),\nu^{j}(X))=0$ for any $(i,j)$ in $\mathcal{R}_{1}$.
\item Let $\mathcal{R}:=\{(i,j)\mid i,j$ integers, $i\neq j, 0\leq i,j<s\}$ and $\mathcal{R}'$=$\mathcal{R}\backslash \mathcal{R}_{1}$. It is easy to see that $\ell({\nu^{i}(X)},\nu^{j}(X))\\\geq4(2m-1)$ for any $(i,j)$ in $\mathcal{R}'$, which is greater than $6m-3$. It follows that $\stHom_{A}(\nu^{i}(X),\nu^{j}(X))=0$ for any $(i,j)$ in $\mathcal{R}'$.
\end{enumerate}

(3) From the description of the support of a low vertex,  if $m\leq q<3m-1$, then $\nu^{e}(X)=\tau^{-(2m-1)}(X)=(2m-1+p,q)\in$ Supp$(\Hom_{k(\mathbb{Z}D_{3m})}((p,q),-))$. It follows that $X$ is not a stable brick for the case $\{(D_{3m},1/3,1)\}$ and $\stHom_{A}(X,\nu^{e}(X))\ncong 0$ for  the case $\{(D_{3m},s/3,1)\mid m,s\in \mathbb{N}, m,s\geq2,3\nmid s\}$.
\end{proof}

We now use the above three lemmas to prove the following result, which plays a key role in proving Theorem  \ref{sms-sufficient-condition}.
\begin{Lem}\label{two-side-orthogonal} Let $A$ be an $RFS$ algebra and $\mathcal{S}$ a family of objects in $A$-$\stmod$. If $\mathcal{S}$ satisfies the three conditions in Theorem \ref{sms-sufficient-condition}, then $\mathcal{^{\perp}S^{\perp}}=\{0\}$.
\end{Lem}

\begin{proof}
First we choose a $\overrightarrow{\Delta}$-line in $\mathbb{Z}\Delta$ for $\Delta$ of types $A_n$ and $D_n$ as the figure (2.1) in Subsection 2.2. Notice that the above $\overrightarrow{\Delta}$-lines are different from the $\overrightarrow{\Delta}$-lines in Lemma \ref{Dn-En-stable-brick}.
Let $\mathcal{T}_{X}$ be the set of  modules in the $\overrightarrow{\Delta}$-line containing $X$ in stable Auslander-Reiten quiver  $_{s}\Gamma_{A}$, and let $\mathcal{T}_{\mathcal{C}}$=$\cup_{c\in\mathcal{C}}$ $\mathcal{T}_{c}$ for a set $\mathcal{C}$ of objects in $A$-$\stind$. Given a set $\mathcal{C}$ of  modules and an indecomposable module $Y$ in $A$-$\stmod$, we say that $\mathcal{C}$ labels $\mathcal{T}_{Y}$ if any object of $\mathcal{T}_{Y}$ is not in $\mathcal{^{\perp}C^{\perp}}$, moreover, a module $X$ in $A$-$\stmod$ labels $\mathcal{T}_{Y}$ if any object of $\mathcal{T}_{Y}$ is not in $^{\perp}X^{\perp}$. There are four cases to be considered.

{\it Case 1.} Type $(A_{n},s/n,1)$ for $n,s\in\mathbb{N}$.

Notice that $A$ is a self-injective Nakayama algebra in this case. The stable Auslander-Reiten quiver   $_{s}\Gamma_{A}$ is of the form $\mathbb{Z}A_{n}/\langle\tau^{-s}\rangle$ for type $(A_{n},s/n,1)$ for $n,s\in\mathbb{N}$. Notice that $s$ is the number of simple modules. We know that the set of vertices in $_{s}\Gamma_{A}$  is the union of $s$ $\overrightarrow{A_{n}}$-lines.
 From the description of support of any indecomposable module in $A$-$\stmod$ (see  Lemma \ref{Riedtmann-lemma}), for an object $X$ in $\mathcal{S}$, any object of $\mathcal{T}_{X}$ is in Supp$(\stHom_{A}(X,-))\ \cup$ Supp$(\stHom_{A}(-,X))$.
 Since $\mathcal{S}$ is an orthogonal system, the $s$ objects in $\mathcal{S}$ label $s$ different $\overrightarrow{A_{n}}$-lines, it follows that $\mathcal{T}_{\mathcal{S}}$ covers the whole stable Auslander-Reiten quiver of $A$, then $\mathcal{^{\perp}S^{\perp}}=\{0\}$ for this type.

 {\it Case 2.} Type $(D_{3m},s/3,1)$ for $m,s\in \mathbb{N},m\geq2, 3\nmid s$.

 We note that the stable Auslander-Reiten quiver $_{s}\Gamma_{A}$ is of the form $\mathbb{Z}D_{3m}/\langle\tau^{-(2m-1)s}\rangle$ for the type $\{(D_{3m},s/3,1)\mid m,s\in\mathbb{N}, m\geq2,3\nmid s\}$, so the set of vertices in $_{s}\Gamma_{A}$ is the union of $(2m-1)s$ different $\overrightarrow{D_{3m}}$-lines. We have the following claim.

{\bf Claim$\colon$} There is precisely one $\nu$-orbit of a  high vertex in  $\mathcal{S}$.  \  \   \  \  \   \  \  \  \   \  \  \   \  \  \   \  \  \  \   \ \  \   \  \  \   \  \  \  \   \ \  \   \  \  \   \  \  \  \   \ \  \   \  \  \   \  \  \  \   \  ($\bigstar$)

Indeed, by Lemma \ref{orthogonal-nu-orbit-of-D_{3m}}, any low vertex $(p,q)$ with $m\leq q<3m-1$ is not in $\mathcal{S}$. Suppose that all objects in $\mathcal{S}$ are not high vertices. Therefore any vertex $Z$ is of the form $(p,q)$ with $p,q$ integers, $1\leq q<m$.
For a vertex $Z=(p,q)$ with $1\leq q<m$, from the description of support of a low vertex, we know that $\mathcal{T}_{Z}$ and $\mathcal{T}_{Z_{1}}$ are in  Supp$(\stHom_{A}(-,Z))\ \cup$ Supp$(\stHom_{A}(Z,-))$, where $Z_{1}=(p+q-3m+1,1)$ (see the picture below).
 Therefore the  vertex $Z$ labels $\mathcal{T}_{Z}$ and $\mathcal{T}_{Z_{1}}$. Since $2m-1\nmid p-(p+q-3m+1)$, $\mathcal{T}_{Z}$ and $\mathcal{T}_{Z_{1}}$ are different $\overrightarrow{D_{3m}}$-lines in $_{s}\Gamma_{A}$.
And since $\mathcal{S}$ is an orthogonal system, $ms$ objects in $\mathcal{S}$ label  $2ms$ different $\overrightarrow{D_{3m}}$-lines in $_{s}\Gamma_{A}$. It contradicts  the fact that $_{s}\Gamma_{A}$ only has $(2m-1)s$ different $\overrightarrow{D_{3m}}$-lines. Therefore by Lemma \ref{orthogonal-nu-orbit-of-D_{3m}}, there is at least one $\nu$-orbit of high vertex in $\mathcal{S}$.

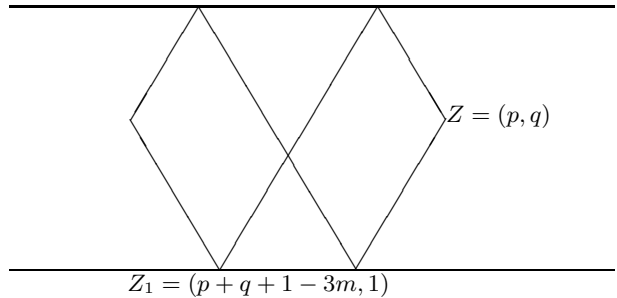
\begin{small}
\begin{center}
\vspace{-1.25cm}
\setlength{\unitlength}{0.35cm}
\begin{picture}(15,14)
 \put(-3,0){\line(1,0){23}}
\put(-3,10){\line(1,0){23}}
\put(5,0){\line(3,5){6}}
\put(5,0){\line(-3,5){3.41}}
\put(4.2,10){\line(-3,-5){2.58}}
\put(4.2,10){\line(3,-5){6}}
\put(10.15,0){\line(3,5){3.45}}
\put(11,10){\line(3,-5){2.57}}
\put(13.6,5.60){$Z=(p,q)$}
\put(1.5,-0.8){$Z_1=(p+q+1-3m,1)$}
\end{picture}
\end{center}
\end{small}
\vspace{0.3cm}

We assume that one of the high vertices is given by $X=(p,3m-1)$ (resp. $(p,3m)$) for some integer $p$. From the proof of Lemma \ref{orthogonal-nu-orbit-of-D_{3m}}, there is a smallest positive integer $e$ such that $\nu^{e}\cong\tau^{-(2m-1)}$, therefore $\nu^{e}(X)=(2m-1+p ,3m-1)$ (resp. $(2m-1+p ,3m)$).
Since $p$ and $2m-1+p$ do not have the same parity and by the description of  support of a high vertex, any high vertex between $X$ and $\nu^{e}(X)$ in $_{s}\Gamma_{A}$, which is $(k,3m-1)$ or $(k,3m)$ with $\overline{p}<\overline{k}<\overline{2m-1+p}$ for some integer $k$, is in Supp$(\stHom_{A}(X,-))\ \cup\ $Supp$(\stHom_{A}(-,\nu^{e}(X)))$.
Suppose now that there is a $\nu$-orbit $O_{\nu}(W)$ of a high vertex $W$ in $\mathcal{S}\backslash O_{\nu}(X)$. Then there is a high vertex $W'$ in $O_{\nu}(W)$  which is between $X$ and $\nu^{e}(X)$ in $_{s}\Gamma_{A}$. It is a contradiction.
Hence there is precisely one $\nu$-orbit of a high vertex in $\mathcal{S}$.

Now we show that $\mathcal{^{\perp}S^{\perp}}=\{0\}$. It suffices to show that the $ms$ objects in $\mathcal{S}$ label $(2m-1)s$ different $\overrightarrow{D_{3m}}$-lines in $_{s}\Gamma_{A}$.
Let $O_{\nu}(X)$ be the only $\nu$-orbit for high vertex $X=(p,3m)$ (resp.  $(p,3m-1)$) for some integer $p$ in $\mathcal{S}$. Since the number of elements in $O_{\nu}(X)$  is $s$, the number of low vertices in  $\mathcal{S}$ is $(m-1)s$.
From the proof of the first half part of ($\bigstar$),  $(m-1)s$ low vertices label $2(m-1)s$ different $D_{3m}$-lines in $_{s}\Gamma_{A}$.
From the proof of the second half  part of ($\bigstar$),
we have $\stHom_{A}(X',\nu^{e}(X))\neq 0$ by the covering theory and the description of support of a high vertex, where $X'=(p,3m)$ (resp. $(p,3m-1)$).
Therefore $X'\in$ Supp$(\stHom_{A}(-,\nu^{e}(X)))$, and $O_{\nu}(X)$ labels $\mathcal{T}_{X}$, we can see that the $s$ objects in $O_{\nu}(X)$ label $s$ different $\overrightarrow{D_{3m}}$-lines in $_{s}\Gamma_{A}$.
Since $\mathcal{S}$ is an orthogonal system,
the $s$ $\overrightarrow{D_{3m}}$-lines corresponding to the $\nu$-orbit for a high vertex $X$ and the $2(m-1)s$ $\overrightarrow{D_{3m}}$-lines corresponding to low vertices are different. Therefore $\mathcal{^{\perp}S^{\perp}}=\{0\}$.
 
{\it Case 3.} For the other types in standard case.

 We assume that $\mathcal{^{\perp}S^{\perp}}\neq\{0\}$.
Take a nonzero indecomposable module $M\in \mathcal{^{\perp}S^{\perp}}$,
and let $\mathcal{S}_{1}:=\mathcal{S}\cup \{O_{\nu}(M)\}$. By Lemma \ref{orthogonal-nu-orbit}, $\mathcal{S}_{1}$ is a Nakayama-stable orthogonal system in $A$-$\stmod$. If $^{\perp}{\mathcal{S}_1}^{\perp}\neq\{0\}$, then we proceed the above step. Since $A$ is of finite representation type,
 there is a positive integer $p$ such that  $^{\perp}{\mathcal{S}}_p^{\perp}=\{0\}$. In particular, $\mathcal{S}_p$ is a Nakayama-stable orthogonal system containing $\mathcal{S}$ properly in $A$-$\stmod$.
 By Lemma \ref{two-torsion-pairs}, both $(\mathcal{^\perp {S}}_p, \mathcal{F}(\mathcal{S}_p))$ and $(\mathcal{F}(\mathcal{S}_p),\mathcal{S}_p^{\perp})$ are torsion pairs. Let $Y$ be an  object in $A$-$\stmod$. Consider a minimal $(\mathcal{^\perp {S}}_p, \mathcal{F}(\mathcal{S}_p))$-triangle,
$$\mathbf{a}Y\longrightarrow Y\longrightarrow \mathbf{b}Y\longrightarrow \mathbf{a}Y[1].$$
 Suppose that $0\neq Y\in\mathcal{S}_p^{\perp}$, by Lemma \ref{Nakayaka-stable-conclusion-2}, $\mathbf{a}Y\in {^{\perp}}{\mathcal{S}}_p^{\perp}=\{0\}$. Therefore $\mathbf{a}Y=0$ and $Y\cong\mathbf{b}Y\in  \mathcal{F}(\mathcal{S}_p)$.
Since $Y\in\mathcal{S}_p^{\perp},$ $\stHom_{A}(Y,Y)=0$, then $Y=0$, it is a contradiction.
Since $(\mathcal{F}(\mathcal{S}_p)$, $\mathcal{S}_p^{\perp})$ is a torsion pair,  $\mathcal{F}(\mathcal{S}_p)= A$-$\stmod$ and $\mathcal{S}_p$ is an sms. By the necessary conditions on $\mathcal{S}$ to be an sms from Subsection 2.3, the number of objects of  $\mathcal{S}_p$  must be the number of non-isomorphic simple modules. Then $\mathcal{S}_p = \mathcal{S}$ and it contradicts  our assumption $\mathcal{S}_p \supsetneq \mathcal{S}$. Hence $^{\bot}\mathcal{S}^{\bot} = {0}$.

{\it Case 4.} Type $(D_{3m},1/3,1)$ ($m\geq2$) in non-standard case.

Recall from Remark \ref{standard-non-standard-pair} that for a non-standard RFS algebra $A$ of type $(D_{3m},1/3,1)$, there is a standard counterpart $A_s$ with the same type, such that there is a bijection $A$-$\stind$ $\leftrightarrow$ $A_s$-$\stind$ between the set of indecomposable objects and irreducible morphisms, which is compatible with the position on the stable Auslander-Reiten quiver
$\Z D_{3m}/\langle \tau^{2m-1}\rangle$. Moreover, there are well-behaved functors $F\colon k(\Z D_{3m})\rightarrow$ $A$-$\stind$ and $F_s\colon k(\Z D_{3m})\rightarrow$ $A_s$-$\stind$ which are compatible with the above bijection $A$-$\stind$ $\leftrightarrow$ $A_s$-$\stind$ (cf. \cite[Section 4]{CKL}). Let $\mathcal{S}$ be  an orthogonal system in $A_s$-$\stind$. Since the well-behaved functors are covering functors, by the formulas on covering functors in Subsection 2.2, $F( F_s^{-1}(\mathcal{S}))$ is an orthogonal system in $A$-$\stind$. Similarly, if $\mathcal{S'}$ is an  orthogonal system in $A$-$\stind$, then $F_s( F^{-1}(\mathcal{S'}))$ is an  orthogonal system in $A_s$-$\stind$. This implies that there is a one to one correspondence between the set of orthogonal systems in $A$-$\stind$ and that in $A_s$-$\stind$. It follows that we also have $\mathcal{^{\perp}S^{\perp}}=\{0\}$ in this case.
\end{proof}

\begin{proof}[{\bf Proof of Theorem 3.1}]
 Necessity is clear from Subsection 2.3, so now we assume that $\mathcal{S}$ satisfies the conditions $(1),(2)$ and $(3)$. Since $\mathcal{S}$ is an orthogonal system, by Lemma \ref{two-torsion-pairs}, both $(\mathcal{^\perp S}, \mathcal{F}(\mathcal{S}))$ and $(\mathcal{F}(\mathcal{S}),\mathcal{S^{\perp}})$ are torsion pairs. Let $Y$ be an  object in $A$-$\stmod$.
 Consider a minimal $(\mathcal{^\perp S}, \mathcal{F}(\mathcal{S}))$-triangle,
$$\mathbf{a}Y\longrightarrow Y\longrightarrow \mathbf{b}Y\longrightarrow \mathbf{a}Y[1].$$
Suppose that $0\neq Y\in\mathcal{S^{\perp}}$, by
 Lemma \ref{Nakayaka-stable-conclusion-2}, $\mathbf{a}Y\in\mathcal{^{\perp}S^{\perp}}$. According to Lemma \ref{two-side-orthogonal}, $\mathbf{a}Y=0$ and $Y\cong \mathbf{b}Y\in  \mathcal{F}(\mathcal{S})$. Since  $Y\in\mathcal{S^{\perp}},$ this implies that $\stHom_{A}(Y,Y)=0$, and so $Y=0$, it is a contradiction, therefore $\mathcal{S^{\perp}}=\{0\}$.
Since $(\mathcal{F}(\mathcal{S})$, $\mathcal{S^{\perp}})$ is a torsion pair, we have $\mathcal{F}(\mathcal{S})= A$-$\stmod$ and therefore $\mathcal{S}$ is an sms.
\end{proof}

We have the following immediate consequence from the proof of Theorem \ref{sms-sufficient-condition}.

\begin{Cor}\label{symmetric-case}Let $A$ be an RFS algebra and $\mathcal{S}$ a Nakayama-stable orthogonal system in $A$-$\stmod$. If  $\mathcal{^{\perp}S^{\perp}}=\{0\}$, then $\mathcal{S}$ is an sms.
\end{Cor}

\subsection{Extendible Nakayama-stable orthogonal systems} In this subsection, we prove the following extendible property of Nakayama-stable orthogonal systems for RFS algebras.

\begin{Thm}\label{orthogonal-system-RFS-extend-to-sms} Let $A$ be an RFS algebra. Then every Nakayama-stable orthogonal system $\mathcal{S}$ in $A$-$\stmod$ extends to an sms.
\end{Thm}

\begin{Proof} This is a consequence of the following three lemmas$\colon$ \ref{orthogonal-system-extend-to-sms}, \ref{orthogonal-system-of-A_n-extend-to-sms-1}, \ref{orthogonal-system-of-D_n-extend-to-sms-1}.

\end{Proof}

\begin{Lem}\label{orthogonal-system-extend-to-sms} Let $A$ be an RFS algebra of type $(A_{2p+1},s,2)$, $D_{n}$ {\rm(except $\{(D_{3m},s/3,1)$ with $ m,s\in \mathbb{N},m\geq2,3\nmid s\}$)} or $E_{n}$. Then every Nakayama-stable orthogonal system $\mathcal{S}$ in $A$-$\stmod$ extends to an sms.
\end{Lem}
\begin{proof}
By Corollary \ref{symmetric-case}, if $\mathcal{^{\perp}S^{\perp}}=\{0\}$, then $\mathcal{S}$ is an sms. Otherwise $\mathcal{^{\perp}S^{\perp}}\neq\{0\}$, take a nonzero indecomposable module $M\in \mathcal{^{\perp}S^{\perp}}$.
Let $\mathcal{S}_{1}:=\mathcal{S}\cup \{O_{\nu}(M)\}$. By Lemma \ref{orthogonal-nu-orbit}, $\mathcal{S}_{1}$ is a Nakayama-stable orthogonal system in $A$-$\stmod$. If $^{\perp}{\mathcal{S}_1}^{\perp}=\{0\}$, then $\mathcal{S}_{1}$ is an sms.
Otherwise $^{\perp}{\mathcal{S}_1}^{\perp}\neq \{0\}$, and we can similarly get a Nakayama-stable orthogonal system $\mathcal{S}_{2}$ containing $\mathcal{S}_{1}$ properly. Repeat the above process, since $A$ is of finite representation type,
 there is a positive integer $q$ such that  $^{\perp}{\mathcal{S}_q}^{\perp}=\{0\}$, and $\mathcal{S}_{q}$ is an sms.
\end{proof}

The next lemma deals with RFS algebras of type $\{(A_{n},s/n,1)\mid n,s\in\mathbb{N}\}$, that is, the self-injective Nakayama algebras. We first recall some notations and results on self-injective Nakayama algebras (cf. \cite[Section V]{ASS}).

The self-injective Nakayama algebra $A$ with $s$ simple modules and Loewy length $n+1$ is defined by the following quiver $Q$
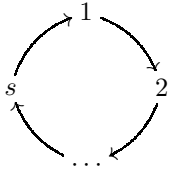
$$\xymatrix@rd{
1\ar@/^/[r]&2\ar@/^/[d]\\
s\ar@/^/[u]&\cdots\ar@/^/[l] }$$ with admissible ideal $I=rad^{n+1}(kQ)$.
Notice that if $A$ is symmetric Nakayama, then $n=ms$ for some positive integer $m$.
 Let $X_1,X_2,\cdots,X_s$ be all the simple $A$-modules. Then $\tau X_{i}=X_{\overline{i+1}}$, where $\tau$ is the AR-translate and $\overline{i+1}$ denotes the positive integer in $\{1,\ldots,s\}$ with $i+1\equiv\overline{i+1}$ mod $s$. Notice that any indecomposable $A$-module $M$ is uniserial and completely determined up to isomorphism by its socle $soc(M)$ and its Loewy length $\ell(M)$. We denote an indecomposable $A$-module $M$ by $X_i(m)$, if $soc(M)$ is isomorphic to $X_i$ and $\ell(M)$ is $m$. We set $X_i(0)=0$ for all $1\leq i\leq s$. It is easy to verify that if $A$ is symmetric Nakayama, then $M$ is a stable brick in $A$-$\stmod$ if and only if $\ell(M)\leq s$ or $n+1-s\leq\ell(M)\leq n$, where the second inequality $n+1-s\leq \ell(M) \leq n$ is a consequence of the corresponding homomorphism in $A$-mod factoring through a projective-injective module.

From the picture of the AR-quiver $\Gamma_{A}$ of the above self-injective Nakayama algebra $A$ ($\cong \mathbb{Z}A_{n+1}/\langle\tau^{-s}\rangle$), we can easily read the following short exact sequence in $A$-mod (for $1\leq i\leq s$, $0< k\leq r\leq n$ and $1\leq j\leq n+1-r$)$\colon$
\begin{equation}\label{ses}
 0\rightarrow X_{i}(r)\xrightarrow{\begin{scriptsize}\begin{pmatrix}\epsilon_{1}\\ \pi_{1}  \end{pmatrix}\end{scriptsize}} X_{i}(r+j)\oplus \tau^{-k}(X_{i}(r-k)) \xrightarrow{\begin{scriptsize}\begin{pmatrix}\pi_{2},\epsilon_{2} \end{pmatrix}\end{scriptsize}} \tau^{-k} (X_{i}(r-k+j))\rightarrow 0,
 \vspace{0.1cm}
\end{equation}
where $\epsilon_{m},\pi_{m}$ for $m\in\{1,2\}$ are the compositions of irreducible maps in the sectional paths of $\Gamma_{A}$. The sequence $(3.1)$ induces the following non-split triangle in $A$-$\stmod$$\colon$
\begin{equation}\label{ses}
 X_{i}(r)\xrightarrow{\begin{scriptsize}\begin{pmatrix}\underline{\epsilon_{1}}\\ \underline{\pi_{1}}  \end{pmatrix}\end{scriptsize}} X_{i}(r+j)\oplus \tau^{-k}(X_{i}(r-k)) \xrightarrow{\begin{scriptsize}\begin{pmatrix}\underline{\pi_{2}},\underline{\epsilon_{2}} \end{pmatrix}\end{scriptsize}} \tau^{-k} (X_{i}(r-k+j))\rightarrow \Omega^{-1}(X_{i}(r)).
 \vspace{0.2cm}
\end{equation}
Notice that in the above triangle, the Loewy lengths of all modules are less than or equal to $n+1$, and we treat $X_i(n+1)=0$  in $A$-$\stmod$ for each $i$ since $X_{i}(n+1)$ is an indecomposable projective-injective module.

\begin{Lem}\label{orthogonal-system-of-A_n-extend-to-sms-1} Let $A$ be an RFS algebra of type $\{(A_{n},s/n,1)\mid n,s\in\mathbb{N}\}$  {\rm(that is, $A$ is a self-injective Nakayama algebra with $s$ simple modules and Loewy length $n+1$)} and $\mathcal{S}$  a Nakayama-stable orthogonal system in $A$-$\stmod$. Then $\mathcal{S}$ extends to an sms.
\end{Lem}

\begin{proof} Let $A$ be as above and $B$ be the symmetric Nakayama algebra with $e$ simple modules and Loewy length $n+1$, where $e$ is the greatest common divisor of $s$ and $n$. Then there is a covering of stable translation quivers $\pi\colon$$_{s}\Gamma_{A}\longrightarrow$ $_{s}\Gamma_{B}\cong$ $_{s}\Gamma_{A}/\langle \nu\rangle$ (where $\nu$ is the Nakayama automorphism of $_{s}\Gamma_{A}$), which induces a covering functor $F\colon$$A$-$\stmod\longrightarrow B$-$\stmod$ (cf.
\cite[Lemma 4.15]{GLYZ}). Consequently, if $\mathcal{S}$ is an orthogonal system in $B$-$\stmod$, then $\mathcal{S}$ is an sms of $B$-$\stmod$ if and only if $F^{-1}(\mathcal{S})$ is an sms of $A$-$\stmod$ (cf.
\cite[Lemma 4.15]{GLYZ}). Therefore, without loss of generality, we can assume that $A$ is a symmetric Nakayama algebra and $n=ms$. If $m=2$, then all indecomposable modules are stable bricks. By the proof of  Corollary \ref{orthogonal-system-extend-to-sms}, $\mathcal{S}$ extends to an sms. Therefore we can assume that $m>2$.

If $\mathcal{^{\perp}S^{\perp}}=\{0\}$, then by Corollary \ref{symmetric-case}, $\mathcal{S}$ is an sms.

Otherwise $\mathcal{^{\perp}S^{\perp}}\neq\{0\}$, we can take a nonzero indecomposable module $X_i(as+b)\in \mathcal{^{\perp}S^{\perp}}$ for some positive integers $0\leq a\leq m-1$ and $1\leq b\leq s$.
If  $X_i(as+b)$ is a stable brick, that is, $a=0$ and $1\leq b\leq s$
 (or $a=m-1$ and $1\leq b\leq s$), then let $\mathcal{S}_{1}:=\mathcal{S}\cup \{X_i(as+b)\}$.

 If $X_i(as+b)$ is not a stable brick, then $1\leq a\leq m-2$ and $1\leq b\leq s$, and there are two cases to be considered.

{\it Case 1.}  $n\geq 2as+b.$ Consider the following triangle
\begin{equation}\label{ses1}
 X_{i}(as+b)\xrightarrow{\begin{scriptsize}\begin{pmatrix}\underline{\epsilon_{1}}\\ \underline{\pi_{1}}  \end{pmatrix}\end{scriptsize}} X_{i}(2as+b)\oplus X_{i}(b) \xrightarrow{\begin{scriptsize}\begin{pmatrix}\underline{\pi_{2}},\underline{\epsilon_{2}} \end{pmatrix}\end{scriptsize}}  X_{i}(as+b)\rightarrow \Omega^{-1}(X_{i}(as+b))
 \vspace{0.2cm}
\end{equation}
by taking $r=as+b,k=as$ and $j=as$ in (\ref{ses}). We know that $X_i(b)$ is a stable brick. Apply \underline{Hom}$_A(S,-)$ and  \underline{Hom}$_A(-,S)$ for all $S\in\mathcal{S}$  to the triangle \eqref{ses1}, we get that $X_i(b)\in \mathcal{^{\perp}S^{\perp}}$. Let $\mathcal{S}_{1}:=\mathcal{S}\cup \{X_i(b)\}$.

{\it Case 2.}  $n<2as+b.$ Consider the following triangle
\begin{equation*}
 X_{i}(as+b)\xrightarrow{\begin{scriptsize}\begin{pmatrix}\underline{\epsilon_{1}}\\ \underline{\pi_{1}}  \end{pmatrix}\end{scriptsize}} X_{i}(n-s+b)\oplus X_{i}((2a+1)s-n+b) \xrightarrow{\begin{scriptsize}\begin{pmatrix}\underline{\pi_{2}},\underline{\epsilon_{2}} \end{pmatrix}\end{scriptsize}}  X_{i}(as+b)\rightarrow \Omega^{-1}(X_{i}(as+b))
 \vspace{0.2cm}
\end{equation*}
by taking $r=as+b,k=n-(a+1)s$ and $j=n-(a+1)s$ in (\ref{ses}). Notice that $(2a+1)s-n+b> s+1$ when $n<2as+b.$ Similarly to Case 1, we get that $X_{i}(n-s+b)$ is a stable brick and $X_{i}(n-s+b)\in \mathcal{^{\perp}S^{\perp}}$. Let $\mathcal{S}_{1}:=\mathcal{S}\cup \{X_{i}(n-s+b)\}$.

From the above discussion, in any case we get a Nakayama-stable orthogonal system $\mathcal{S}_{1}$ containing $\mathcal{S}$ properly. Repeat the above process, since $A$ is of finite representation type,
 there is a positive integer $q$ such that  $^{\perp}{\mathcal{S}_q}^{\perp}=\{0\}$.  By Corollary \ref{symmetric-case}, $\mathcal{S}_{q}$ is an sms.
\end{proof}

The last lemma deals with the remaining RFS algebras, that is, the RFS algebras of type $(D_{3m},s/3,1)$ with $m\geq2,3\nmid s$.

\begin{Lem}\label{orthogonal-system-of-D_n-extend-to-sms-1} Let $A$ be an RFS algebra of type $\{(D_{3m},s/3,1)\mid m\geq2,3\nmid s\}$ and $\mathcal{S}$  a Nakayama-stable orthogonal system in $A$-$\stmod$. Then $\mathcal{S}$ extends to an sms.
\end{Lem}
\begin{proof}
By covering theory and the standard-non-standard correspondence (cf. \cite[Section 4]{CKL} and the proof of Lemma \ref{two-side-orthogonal}), we only need  to consider the standard RFS algebras of type $(D_{3m},1/3,1)$. Notice that in this case the algebras are symmetric and the orthogonal systems in $A$-$\stmod$ are automatically Nakayama-stable. Now suppose that $A$ is an RFS algebras of type $(D_{3m},1/3,1)$ and $\mathcal{S}$ is an orthogonal system in $A$-$\stmod$. Then the stable AR-quiver $_s\Gamma_A$ has the form $\mathbb{Z}D_{3m}/ \langle\tau^{-(2m-1)}\rangle$ and the stable category $A$-$\stmod$ is determined by the mesh category $k(_{s}\Gamma_{A})$.

In the following proof, we often identify the indecomposable objects in $A$-$\stmod$ with vertices in $_s\Gamma_A$. We first show that if $\mathcal{S}$ contains a high vertex $S_{0}$, then $\mathcal{S}$ extends  to an sms. Without loss of generality, let $S_{0}=(1,3m)$ be a high vertex in $\mathcal{S}$. By the description of support of a high vertex in $k(\mathbb{Z}D_{3m})$,
$$^{\perp}S_{0}^{\perp}=\{ (i,j) \mid m+1\leq i\leq 2m-1, i+j\leq 2m \}\cup \{ (i,j) \mid 2\leq i\leq m-1, i+j\leq m \}.$$
It follows that if $X=(i,j)$ is in  $^{\perp}S_{0}^{\perp}$, then we have  $1\leq j<m$. By Lemma \ref{orthogonal-nu-orbit-of-D_{3m}},
all objects in $^{\perp}S_{0}^{\perp}$ are stable bricks, it follows that all objects in $\mathcal{^{\perp}S^{\perp}}$ are all stable bricks. Similarly to the proof of Lemma \ref{orthogonal-system-extend-to-sms}, $\mathcal{S}$ extends to an sms.

Next we claim that any orthogonal system $\mathcal{S}$ in $A$-$\stmod$ extends to an orthogonal system which contains a high vertex. We can assume that $\mathcal{S}$ does not contain any high vertex. If $\mathcal{^{\perp}S^{\perp}}=\{0\}$, by Corollary \ref{symmetric-case}, then $\mathcal{S}$ is an sms. By Case 2 of Lemma \ref{two-side-orthogonal}, $\mathcal{S}$
contains a unique high vertex, which is a contradiction. Therefore $\mathcal{^{\perp}S^{\perp}}\neq\{0\}$ and we consider two cases.

 {\it Case 1.}  All indecomposable objects in $\mathcal{^{\perp}S^{\perp}}$ are stable bricks. Similarly to the proof of Lemma
\ref{orthogonal-system-extend-to-sms},
we can extend $\mathcal{S}$ to an sms with a unique high vertex.

{\it Case 2.} There is an indecomposable object $X$ in $\mathcal{^{\perp}S^{\perp}}$ which is not a stable brick.
Without loss of generality, we can assume that $X=(1,t)$ for some $ m\leq t<3m-1$.
By the description of support of vertex
in $k(\mathbb{Z}D_{3m})$, there are three subcases to be considered (where $S_0=(1,3m)$)$\colon$
\begin{enumerate}[(i)]
\item If $m+1\leq t\leq 2m-1$, then
$^{\perp}X^{\perp}={^{\perp}S_{0}^{\perp}}\setminus(\{(i,j)\mid t+1\leq i+j\leq 2m, m+1\leq i\leq t  \}
\cup \{(i,j)\mid t-m+2\leq i+j\leq m , 2\leq i\leq t-m+1 \}).$
\item If $2m+1\leq t\leq 3m-2$, then
$^{\perp}X^{\perp}={^{\perp}S_{0}^{\perp}}\setminus(\{(i,j)\mid t-2m+2\leq i+j\leq m, 2\leq i\leq t-2m+1  \}
\cup \{(i,j)\mid t-m+2\leq i+j\leq 2m , m+1\leq i\leq t-m+1 \}).$
\item If $t=m$ or $2m$, then $^{\perp}X^{\perp}={^{\perp}S_{0}^{\perp}}$.
\end{enumerate}
Therefore, $\mathcal{S}\subseteq$ $^{\perp}X^{\perp}\subseteq {^{\perp}S_{0}^{\perp}}$ and $S_{0}\in \mathcal{^{\perp}S^{\perp}}$. This shows that we can add the high vertex $S_0$ to $\mathcal{S}$ and reduce the proof to Case 1.
\end{proof}

Finally, we assume that $A$ is a representation-finite symmetric algebra. For any idempotent
element $e$ in $A$,  $eAe$ is also a representation-finite symmetric algebra and the
idempotent embedding functor $\iota\colon eAe$-$\stmod\hookrightarrow A$-$\stmod$
is fully faithful (see \cite[Page 12]{M}). Thus we get the following corollary of Theorem \ref{orthogonal-system-RFS-extend-to-sms}.

\begin{Cor}\label{idempotent-extend-to-sms} Let $A$ be a representation-finite symmetric algebra and $e$  an idempotent element of $A$. If $\mathcal{S}$ is an sms in $eAe$-$\stmod$, then $\iota(\mathcal{S})$ extends to an sms in $A$-$\stmod$.
\end{Cor}

\end{document}